\documentclass[12pt, a4paper]{amsart}

\usepackage{amsmath, amsthm, amssymb}
\usepackage{graphicx}
\usepackage{enumerate}
\usepackage[top=2.5cm, bottom=2.5cm, left=2.5cm, right=2.5cm]{geometry}

\newcommand{\x}{\mathbf x}

\newtheorem{lemma}{\bf Lemma}
\newtheorem{proposition}{\bf Proposition}
\newtheorem{theorem}{\bf Theorem}
\newtheorem{corollary}{\bf Corollary}

\renewenvironment{proof}{\noindent {\bf Proof: }}{\rm\\}
\theoremstyle{definition}
\newtheorem{remark}{Remark}{\rm}
{\rm}

\usepackage{color}
\usepackage{stmaryrd}

\usepackage[T1]{fontenc}
\usepackage[export]{adjustbox}

\usepackage{algorithm, algorithmic}

\makeatletter
\renewcommand{\p@algorithm}{\arabic{algorithm}\expandafter\@gobble}
\makeatother

\newcounter{step}[algorithm]
\newcommand\STEP[2][\(\triangleright\)]{%
	\refstepcounter{step}
	\vskip 0.25\baselineskip
	\item[]\hskip -\algorithmicindent #1 \textbf{Step \arabic{step}}%
	\ifthenelse{\equal{\unexpanded{#2}}{}}{}{ (\texttt{#2})}%
	\textbf{.}%
}
\newenvironment{algo}{\algo}{}
\def\algo#1\end{%
	\noindent\fbox{%
	\begin{minipage}[b]{\dimexpr\columnwidth-\algorithmicindent\relax}
	\begin{algorithmic}
	#1
	\end{algorithmic}
	\end{minipage}
	}%
\end}

\newcommand{\afunc}[1]{\operatorname{\mathsf{#1}}}

\usepackage{mathrsfs}

\begin{document}

\title[Control of an anti-stable wave]{Boundary feedback control of an anti-stable wave equation}
\author{Pierre APKARIAN$^1$}
\author{Dominikus NOLL$^{2}$}
\thanks{$^1$ONERA, Department of System Dynamics, Toulouse, France}
\thanks{$^2$Institut de Math\'ematiques, Universit\'e de Toulouse, France}
\date{}

\begin{abstract}
We discuss boundary control of a  wave equation with a non-linear anti-damping boundary condition.
We design structured finite-dimensional $H_\infty$-output feedback controllers which stabilize the infinite dimensional system exponentially in closed loop.
The method is applied to control torsional vibrations in drilling systems with the goal
to avoid slip-stick.
\\

\noindent
{\bf Key words:} 
Wave equation $\cdot$ boundary feedback control $\cdot$  anti-damping boundary $\cdot$  infinite-dimensional $H_\infty$-control
$\cdot$ slip-stick $\cdot$ torsional vibrations $\cdot$ large magnitude sector non-linearity\end{abstract}
\maketitle


\section{Introduction}
We discuss $H_\infty$-boundary feedback control of a
wave equation with instability caused by boundary anti-damping. This is applied
to the  control of vibrations in drilling devices.
The system we consider is of the form
\begin{align}
\label{system}
G_{nl}:
\begin{split}
x_{tt}(\xi,t) &= x_{\xi\xi}(\xi,t) - 2 \lambda x_t(\xi,t)  \qquad 0 < \xi < 1, t \geq 0 \\ 
x_\xi(1,t) &= -x_t(1,t) +u(t) \\
\alpha x_{tt}(0,t) &= x_\xi(0,t) + q x_t(0,t) + \psi\left(  x_t(0,t) \right)
\end{split}
\end{align}
where $(x,x_t)$ is the state, $u(t)$ the control, and  where the measured outputs are
\begin{align}
\label{outputs}
y_1(t) = x_t(0,t), \quad y_2(t) = x_t(1,t).
\end{align}
The non-linearity $\psi$ satisfies $\psi(0)=0$, $\psi'(0)=0$, and the steady state
is $(x,x_t,u) = (0,0,0)$.  The linearized system $G$ 
is obtained from (\ref{system}) by dropping the term $\psi(x_t(0,t))$.

The parameters satisfy
$\lambda \geq 0$, $\alpha \geq 0$, while $q$ is signed.
System (\ref{system})  was first discussed
in \cite{tucker:99,challamel:00,tucker:03} in the context of oil-well drilling. The author of \cite{challamel:00} proves open-loop stability of (\ref{system}) for the 
case $q < 0$ using Lyapunov's direct method. Since applications typically lead to the opposite case $q > 0$, 
where instability occurs,
various control strategies have been proposed for that setting.  

Lyapunov's direct method is used in
\cite{saldivar:13,saldivar:16,barreau:18b,basturk:17}.  This
leads to infinite-dimensional controllers which are not
implementable,  or to observer-based controllers, which due to their lack of robustness
are out of favor since the late 1970s.

Delay system techniques are used in \cite{fridman:10,saldivar:09,saldivar:13,saldivar:15,cheng:17,apkarian:19}, but  require $\lambda = 0$,
which leads to an oversimplified model.
Input shaping  is used in \cite{pilbauer:18}, but as presented, also requires the un-damped model $\lambda=0$. 

Backstepping control is used in \cite{smyshlyaev:09,roman,bresch_krstic:14,roman_back,sagert,davo}, but with the exception of
\cite{bresch_krstic:14}, where $\lambda=\alpha =0$, leads to infinite dimensional or state feedback controllers, which are not implementable.   
Infinite dimensional controllers can also be obtained with the method in \cite{barreau:18}.
 Other ideas
to avoid slip-stick include
the design of feedforward startup trajectories \cite{aarsnes:18},
or manipulation of the weight on bit  in \cite{saldivar:09,saldivar:13}. 
 Model (\ref{system}), (\ref{outputs})  has also been used to control axial vibrations, see \cite{saldivar:13,beji:17}, and for robotic drilling
 \cite{beji:17}.  
 
What these approaches have in common is that they are guided by the {\it method of proof} of infinite-dimensional
 stability.  This leads to impractical control laws. 
In contrast,
approaches guided by {\it practical considerations}  have also been applied to oil-well drilling \cite{serrarens,besselink}, 
but those use finite-dimensional approximate
models. 
This makes it desirable to bridge between both
approaches by designing  practical controllers using the infinite-dimensional model (\ref{system}).
In the present work we
design $H_\infty$-controllers  with
the following requirements:
\begin{itemize}
\item[(a)] The controller is output feedback and of a simple, implementable structure, like a reduced-order  controller or a PID.
\item[(b)] The controller stabilizes the  infinite-dimensional system (and not just a finite-dimensional approximation of it).
\item[(c)] 
$H_\infty$-optimality of the controller is certified in closed loop with the infinite-dimensional system (and not just with a finite-dimensional approximation).  
\item[(d)] Due to the achieved infinite-dimensional $H_\infty$-performance, 
slip-stick is avoided, or at least mitigated.
\end{itemize}
 
 These requirements are achieved by going through the steps of
the following general  $H_\infty$-control scheme, which we proposed
for boundary and distributed control of PDEs in
 \cite{ANR:17,AN:18,apkarian:19}, where it has already been applied successfully to a variety of applications.

 \begin{algorithm}[!ht]
\caption{Infinite-dimensional $H_\infty$-design \label{algo1}}
\begin{algo}
\STEP{Steady-state} 
Compute steady state of non-linear system $G_{\rm nl}$ and obtain linearization $G$. Compute
transfer function $G(s)$ and determine number $n_p$ of unstable poles of $G$.  
\STEP{Stabilize}
Fix practical controller structure $K(\x)$, and compute initial stabilizing controller $K(\x^0)$ for $G$.
Use Nyquist test to certify stability of linear infinite-dimensional closed loop.
\STEP{Performance}
Determine plant $P$ with $H_\infty$-performance and robustness specifications,  addressing in particular the non-linearity.
\STEP{Optimize}
Solve discretized infinite-dimensional multi-objective $H_\infty$-optimization program using a non-smooth trust region or bundle method \cite{AN:18,ANR:17}.
\STEP{Certificate}
Certify final result in infinite-dimensional system within  pre-specified tolerance level as in \cite{AN:18,ANR:17}. 
\end{algo}
\end{algorithm}

While some of the elements of algorithm \ref{algo1} are standard, others  need to be adapted to
the current case and to be explained in detail. In section \ref{drilling}, the mechanical model for control
$G_{\rm nl}$ will be derived. Its linearization $G$, transfer function, and open-loop properties will be discussed in sections \ref{linear} and \ref{pattern}.
Locally exponentially 
stabilizing controllers will be synthesized in section  \ref{stabilize}, and $H_\infty$-synthesis for the full, non-linear model 
in section \ref{synthesis} will complete the procedure. Numerical results are regrouped in section \ref{numerics}.

\section{Model of drilling system}
\label{drilling}
We derive model (\ref{system}) from the setup of an oil-well drilling system,  shown schematically in Fig. \ref{oil-well}. 
The state of the system is described by the angular position $\theta(\xi,t)$ and angular speed $\theta_t(\xi,t)$
of the drillstring, where
position $\xi = 0$ refers to the rotary table (top), while $\xi = L$ represents the drill bit (bottom hole assembly), with
$L$ the length of the string.
The dynamic equation and boundary conditions are
\begin{align}
\label{slip-stick1}
\begin{split}
GJ \theta_{\xi\xi}(\xi,t) &= I \theta_{tt}(\xi,t) + \beta \theta_t(\xi,t) \quad 0 < \xi < L, t \geq 0 \\
I_B \theta_{tt}(L,t) &= -GJ \theta_\xi(L,t) - \phi\left(  \theta_t(L,t) \right) \\
\theta_t(0,t) &= \frac{GJ}{c_a} \theta_\xi(0,t) + \Omega(t)
\end{split}
\end{align}
where $G$ is the angular shear modulus, $J$ is the geometrical moment of intertia, 
$I$ is the inertia of the string, $I_B$ is the lumped inertia of the bottom hole assembly,
$c_a$ is related to the local torsion of the drillstring, $\Omega(t)$ is the time-dependent rotational velocity 
coming from the rotary table at the top, used to drive and  control the system,
while the undriven bottom extremity (bit) is subject to a
torque $\phi(\theta_t)$ representing rock-bit and mud-bit interaction of the drill bit, depending non-linearly on the rotary speed
$\theta_t(L,t)$ at the bottom; \cite{challamel:00,timoshenko,tucker:99,tucker:03}. The torsional excitations of the drillstring caused by the frictional force $\phi(\theta_t(L,t))$
lead to  twisting of the string, and this effect propagates along the structure from bit to top as a wave with damping
factor $\beta > 0$. Similarly, alterations in the rotary speed $\Omega(t)$ at the top are transmitted to bottom
by the same damped wave. This implies that a control action at the top will be delayed by one period of the wave
before taking effect at the bottom. 
If measurements are taken only at the top, then the delay before a control action takes effect is even two periods.

\begin{figure}[!ht]
\begin{center}
    \includegraphics[scale=1.0,valign=c]{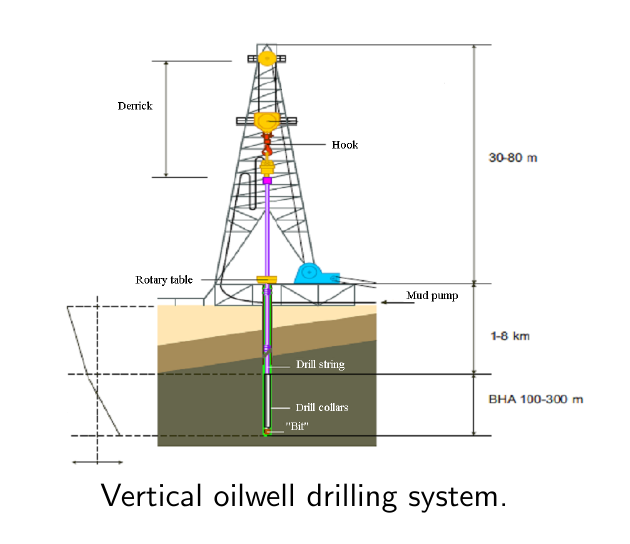}
    \includegraphics[scale=0.6,valign=c]{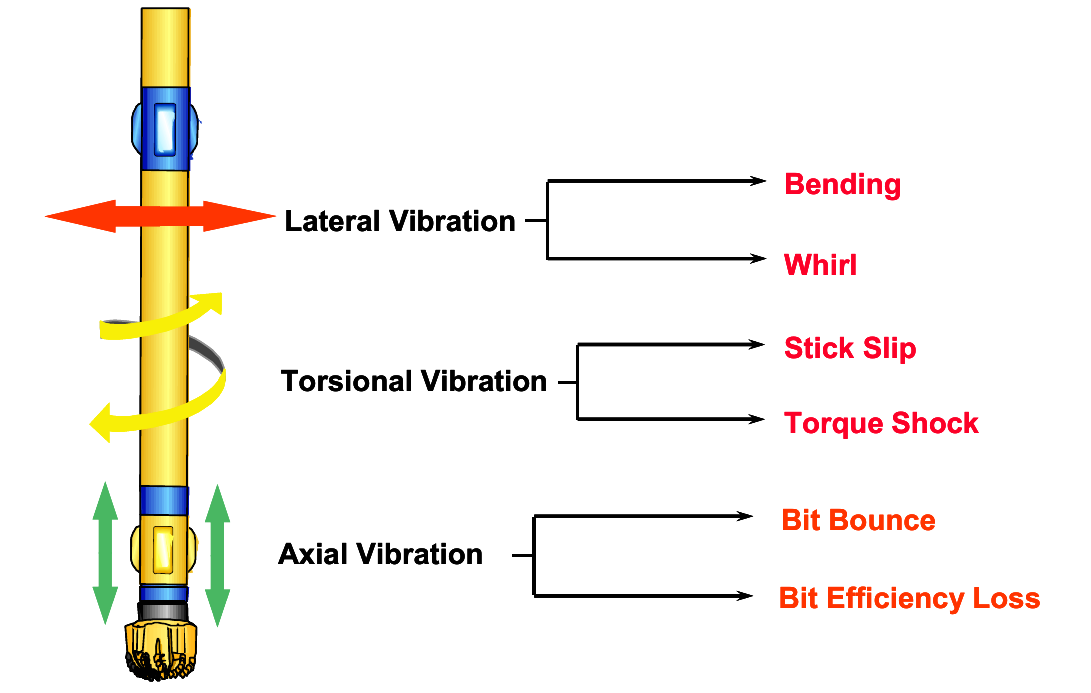}
\end{center}
\caption{ \label{oil-well}}
\end{figure}

The goal of the active control scenario is to maintain the system at steady state with constant 
rotational velocity $\theta_t(L,t) =\Omega$ at the drill bit position $\xi = L$ by acting on the driving rotary force $\Omega(t)$, and using measurements of the rotary speed
at top and bottom.  The steady state solution of (\ref{slip-stick1})  is easily obtained as
\[
\theta^0(\xi,t) = \Omega t - \left( \frac{\phi(\Omega) + \beta \Omega L}{GJ}  \right) \xi + \frac{\beta \Omega}{2GJ} \xi^2
\]
and this corresponds to applying a constant  control torque
$\Omega(t) = \Omega_0$ at the top, where
\[
\Omega_0 = \Omega + \frac{\phi(\Omega)+\beta\Omega L}{c_a}. 
\]
Writing the state in the form $\theta(\xi,t) = \theta^0(\xi,t) + \vartheta(\xi,t)$ for an off-set variable
$\vartheta(\xi,t)$,  and subtracting the steady state from (\ref{slip-stick1}), we obtain the equivalent system 
\begin{align}
\label{slip-stick2}
\begin{split}
GJ\vartheta_{\xi\xi}(\xi,t) &= I\vartheta_{tt}(\xi,t) + \beta \vartheta_t(\xi,t) \\
I_B \vartheta_{tt}(L,t) &= -GJ\vartheta_\xi(L,t) + \phi(\Omega) -  \phi\left(\Omega + \vartheta_t(L,t)\right) \\
GJ \vartheta_\xi(0,t) &= c_a\left(\vartheta_t(0,t) - \Omega(t) +\Omega_0  \right)
\end{split}
\end{align}
A dimensionless system is now obtained by the change of variables
\[
\xi = L(1-\zeta) \quad \tau = \frac{1}{L} \sqrt{\frac{GJ}{I}} t.
\]
On putting $x(\zeta,\tau) = \vartheta(\xi,t)$, this leads to the following equivalent dimensionless form
\begin{align}
\label{slip-stick3}
\begin{split}
x_{\zeta\zeta}(\zeta,\tau) &= x_{\tau\tau}(\zeta,\tau) + \frac{\beta L}{\sqrt{GJI}} x_\tau(\zeta,\tau) \\
\frac{I_B}{LI} x_{\tau\tau}(0,\tau) &= x_\zeta(0,\tau) + \frac{L}{GJ} \left( \phi(\Omega) - \phi\left( \Omega + \frac{1}{L} \sqrt{\frac{GJ}{I}}  x_\tau(0,\tau)\right)   \right) \\
x_\zeta(1,\tau) &= - \frac{c_a}{\sqrt{GJI}} x_\tau(1,\tau) + \frac{c_aL}{GJ} \left( \Omega(t) - \Omega_0  \right)
\end{split}
\end{align}
We re-write the second boundary condition of (\ref{slip-stick3})  at $\zeta=1$ as
\[
x_\zeta(1,\tau) + x_\tau(1,\tau) = \left( 1-\frac{c_a}{\sqrt{GJI}}  \right) x_\tau(1,\tau) + \frac{c_aL}{GJ} \left( \Omega(\tau) - \Omega_0 \right).
\]
Taking into consideration that the measured outputs of (\ref{slip-stick1}) are the angular velocities at the top and bottom positions
$y_1(t) = \theta_t(L,t)$, $y_2(t)= \theta_t(0,t)$, the outputs of the
centered system (\ref{slip-stick3}) may be understood as measurements of the offset angular velocities 
$y_1(\tau) = x_\tau(0,\tau)$ and $y_2(\tau) = x_\tau(1,\tau)$.  This allows us to introduce the control
\[
u(\tau) =  \left( 1-\frac{c_a}{\sqrt{GJI}}  \right) x_\tau(1,\tau) + \frac{c_aL}{GJ} \left( \Omega(\tau) - \Omega_0 \right),
\]
which when chosen in feedback form $u(\tau) = K(y_1(\tau),y_2(\tau))$ leads to the following final feedback control law for
(\ref{slip-stick1}): 
\[
\Omega(t) = \Omega_0 + \left[ K(y_1(t),y_2(t)) + \left( \frac{c_a}{\sqrt{GJI}} -1 \right) y_2(t) \right] \frac{GJ}{c_aL},
\]
which is linear as soon as $u=Ky$ is a linear controller.  With that the second boundary condition takes indeed the form
$x_\zeta(1,\tau) + x_\tau(1,\tau) = u(\tau)$  in (\ref{system}).

Switching back for convenience to $t$ for time and $\xi\in [0,1]$ for the spatial variable, 
and introducing the dimension free parameters
\begin{equation}
\label{parameters}
\alpha = \frac{I_B}{LI}, \quad \lambda = \frac{\beta L}{2 \sqrt{GJI}}, \quad q = -\frac{\phi'(\Omega)}{\sqrt{GJI}},
\end{equation}
system (\ref{slip-stick3}) turns into the form (\ref{system}) with the non-linearity given by
\begin{equation}
    \label{NL}
\psi(\omega) = \frac{L}{GJ} \left( \phi(\Omega) - \phi\left(  \Omega + \frac{1}{L} \sqrt{\frac{GJ}{I}}\omega \right)  \right) - q\cdot \omega,
\end{equation}
with
\[
\psi'(\omega)=-
\frac{\phi'\left(\Omega+\frac{1}{L}\sqrt{\frac{GJ}{I}}\omega\right)}
{\sqrt{GJI}}-q, \quad
\psi''(\omega)=-
\frac{\phi''\left(\Omega+\frac{1}{L}\sqrt{\frac{GJ}{I}}\omega\right)}{LI}.
\]
From the definition of $q$
it can be readily seen that $\psi(0)=0$ and $\psi'(0)=0$, which complies with the requirement in (\ref{system}). For later use, 
we
introduce an additional parameter
\[
p := \psi''(0),
\]
which represents the curvature of the non-linearity at the reference position, and gives information on its severity.

\begin{figure}[!ht]
\includegraphics[scale=1.0]{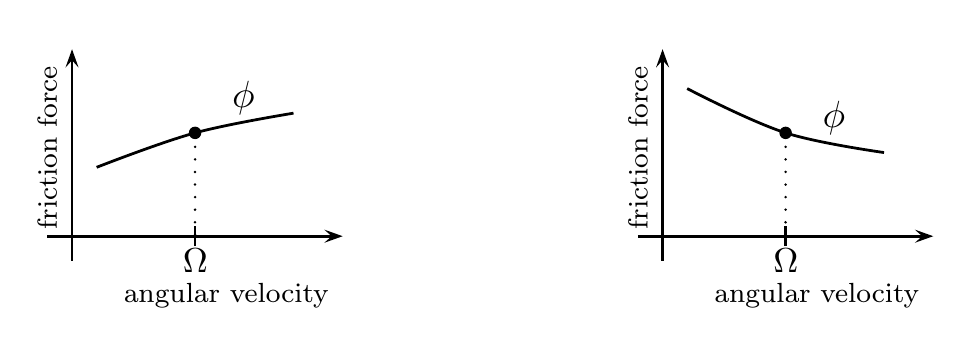}
\caption{The case $\phi'(\Omega)>0$ (left) leads to a stable open loop \cite{challamel:00}. The potentially unstable case (right) is when increasing rotary speed reduces friction.}
\end{figure}

Phenomenological models of the frictional force $\phi(\cdot)$ have been proposed in the literature. For instance \cite{navarro:07}
considers  a model of the form
\[
\phi(\theta_t) = \phi_{\rm mud}(\theta_t) + \phi_{\rm rock}(\theta_t),
\]
where the mud friction is assumed of viscous form $\phi_{\rm mud}(\theta_t) = c_b \cdot \theta_t$, while the rock-bit interaction
is the non-linear 
\begin{equation}
    \label{rock}
\phi_{\rm rock}(\theta_t) = W_{ob} R_b  \left[  \mu_{cb} + \left( \mu_{sb} - \mu_{cb}  \right) e^{-\frac{\gamma_b}{\nu_f}| \theta_t|} \right]    {\rm sign}(\theta_t),
\qquad
\phi_{\rm mud}(\theta_t) = c_b \cdot \theta_t.
\end{equation}
Here $W_{ob}$ is the weight on bit, $R_b$ is the radius of the drill, the non-linear term
features the static and Coulomb friction coefficients $\mu_{sb}, \mu_{cb}\in (0,1)$, while the coefficient $\gamma_b\in (0,1)$ is the velocity decrease rate accounting
for the Stribeck effect. 
The fact that $\mu_{sb} > \mu_{cb}$ is the ultimate reason why
the slip-stick phenomenon may occur.  Namely, for $\Omega > 0$
we have $\phi'(\Omega) = c_b - W_{ob} R_b (\gamma_b/\nu_f) \left( \mu_{sb} - \mu_{cb} \right) e^{-(\gamma_b/\nu_f)\Omega}$, which leads to
\[
q = \frac{-c_b + W_{ob} R_b (\gamma_b/\nu_f) \left( \mu_{sb} - \mu_{cb} \right) e^{-(\gamma_b/\nu_f)\Omega}}{\sqrt{GJI}}
\]
which is typically positive due to dominance of the rock-bit  over the mud-bit interaction. In contrast, the curvature
parameter
\[
p = \psi''(0) = -\frac{1}{LI} \phi''(\Omega) = - \frac{W_{ob}R_b \left( \mu_{sb}-\mu_{cb}  \right)  (\gamma_b/\nu_f)^2 e^{-(\gamma_b/\nu_f)\Omega}}{LI}
\]
is typically negative.

{\small
\begin{table}[h]
\title{Five scenarios of physical parameters}
\begin{tabular}{|| c || c | c | c | c | c | c ||}
\hline \hline
       & gray & blue & magenta & red & green & \\
       \hline\hline
$G$ &   79.3e9    &    79.3e9     &  79.3e9      &   79.3e9       &   79.3e9        & ${\rm N\cdot m^{-2}}$ \\
\hline
$J$ &  1.19e-5     & 1.19e-5         & 1.19e-5       & 1.19e-5         & 1.19e-5          & ${\rm m}^4$ \\
\hline
$I$  &   0.095   &  0.095        & 0.095        & 0.095         &  0.095         & ${\rm kg\cdot m}$ \\
\hline
$I_B$ & 89 & 35.6         &    35.6     &  35.6        &    89       & ${\rm kg\cdot m}^2$ \\
\hline
$L$ &   1172  &   2050       &   1172      &  2050         &  1172         & ${\rm m}$ \\
\hline
$\Omega$& 10 &  10 &   10       &  10         &   10         &${\rm rad\cdot s}^{-1}$ \\
\hline
$c_a$& 2000  & 2000       &  2000        &  2000         &   2000         & ${\rm kg\cdot m}^2\cdot {\rm s}^{-1}$\\
\hline
$\beta$&0.1 & 0.16       & 0.01   / 0.5      & 0.01/0.1          &  0.02         & ${\rm N\cdot s}$ \\
\hline
$W_{ob}$& 97347& 146020.5  &     146020.5     &    146020.5        &     146020.5       & ${\rm N}$ \\
\hline
$R_b$&0.155575&  0.18202275     &   0.2022475      &     0.2333625       &0.2022475& ${\rm m}$ \\
\hline
$\mu_{sb}$&0.8& 0.8 & 0.8          &  0.8          &  0.8           & rad \\
\hline
$\mu_{cb}$&0.5 & 0.5 & 0.5          &  0.5          &    0.5         & rad \\
\hline
$\gamma_b$&0.9&0.1& 0.1        &  0.1           &  0.1            & -- \\
\hline
$\nu_f$&1 & 1&   1              &  1          &  1           & ${\rm rad}\cdot {\rm s}^{-1}$ \\
\hline
$c_b$&0.03&0.03&0.03&0.03&0.03& {\rm N m s}\\
\hline\hline
\end{tabular}
\caption{}
\end{table}
}

{\small
\begin{table}[h]
\title{Derived parameters of five scenarios}\\
\begin{tabular}{|| c || c | c | c | c | c | c ||}
\hline \hline
                      & gray & blue & magenta & red & green & \\
       \hline\hline
$\Omega_0$&   15.02      &    19.75      &      21.94    &    20.50    &   19.13       &   ${\rm rad} \cdot {\rm s}^{-1}$ \\
\hline
$-\frac{\Omega L \sqrt{I}}{\sqrt{GJ}}$ &-3.7186&-6.5044&-3.7186&-6.5044&-3.7186& rad\\
\hline
$\lambda$ &  0.1957         &  0.5477        & 0.9786          &  0.03423      &   0.0391       &  -- \\
\hline
$\alpha$   &      0.7994      &    0.1828      & 0.3197          & 0.1828       &  0.7994        & --\\
\hline
$q$          &    0.0019         &      0.9796     &  1.0885        &  1.2559      &  1.0885        & -- \\
\hline
$p$          &   -0.0048          &    -0.1506      & -0.2927          & -0.1931       &  -0.2927        & -- \\
\hline
$\tau/t$   &     2.6892        & 1.5374          & 2.6892         &  1.5374      &  2.6892        & ${\rm s}^{-1}$  \\
\hline
$n_p$ & 0 & 2  &2  & 1& 1&  \\
\hline
$n_z$ &0   & 4  &2    &22   &  4  & \\
\hline\hline
\end{tabular}
\caption{}
\end{table}
}

\section{Analysis of the linear system $G$}
\label{linear}
In this section we determine the number of unstable poles of the linearization $G$ of $G_{\rm nl}$, as this
will be needed later to assure stability of the closed loop.
This discussion is of independent interest, as in a different context the specific form of the non-linearity $\psi(x_t)$ may be unknown, in which case a
linear parametric robust synthesis in  $q$ may be required.

We recall from \cite{challamel:00} that the  open loop $G_{\rm nl}$ is stable for $\phi' > 0$, and the same is true for its linearization $G$.
This means that we may  concentrate on the potentially instable case 
$\phi' \leq  0$, which means $q \geq  0$.
Our goal is to classify the open loop properties of $G$ as a function of the three
parameters $(q,\alpha,\lambda) \in \mathbb R_+^3$.

Laplace transformation of (\ref{system}) leads to
a family of one-dimensional boundary value problems parametrized by $s\in \mathbb C$:
\begin{align}
\label{laplace}
G:
\begin{split}
x_{\xi\xi} (\xi,s) &= (s^2+2\lambda s) x(\xi,s) \\
x_\xi(1,s) &= -sx(1,s) + u(s) \\
x_\xi(0,s) &= (\alpha s^2 -qs) x(0,s)
\end{split}
\end{align}
which we solve explicitly. With the outputs $y_1(s) = sx(0,s)$, $y_2(s) = s x(1,s)$ from (\ref{outputs})  we obtain
\begin{align}
\label{TF}
G(s)
=
\begin{bmatrix}
\displaystyle
\frac{y_1(s)}{u(s)} \\ \\ \displaystyle\frac{y_2(s)}{u(s)}
\end{bmatrix}
=
 \begin{bmatrix}
 \displaystyle
\frac{1}{\frac{e^\sigma-e^{-\sigma}}{2\sigma} \left[ \frac{\sigma^2}{s} + \alpha s^2 - qs \right] + \frac{e^\sigma+e^{-\sigma}}{2} \left[ \alpha s -q+1 \right]}\\
\\
\displaystyle
\frac{\frac{e^\sigma+e^{-\sigma}}{2} + (\alpha s^2-qs) \frac{e^\sigma-e^{-\sigma}}{2\sigma}}{\frac{e^\sigma-e^{-\sigma}}{2\sigma} \left[ \frac{\sigma^2}{s} + \alpha s^2 - qs \right] + \frac{e^\sigma+e^{-\sigma}}{2} \left[ \alpha s -q+1 \right]}\end{bmatrix},
\quad
\sigma(s): = \sqrt{s^2 + 2 \lambda s}.
\end{align}
We now have to determine the number of unstable poles of (\ref{TF}) as a function of $(q,\alpha,\lambda)\in \mathbb R_+^3$.
Note that $G(s)= [1/d(s) ; n(s)/d(s)]$ is a meromorphic function, with $n(s), d(s)$ in (\ref{TF}) holomorphic,
but its analysis is more complicated than that of a pure delay system 
due to  the damping coefficient $\lambda$ and the consequent appearance of the term $\sigma(s)$.

Annihilating
$d(s) = \left( s + 2\lambda + \alpha s^2-qs  \right) \frac{e^\sigma-e^{-\sigma}}{2\sigma} + (\alpha s -q+1) \frac{e^\sigma+e^{-\sigma}}{2}=0$ 
leads to the
complex equation
\begin{equation}
\label{poles1}
q - \alpha s -1 = \frac{2\lambda} {s +  \frac{(e^\sigma+e^{-\sigma})/2}{(e^\sigma-e^{-\sigma})/2\sigma}  } =: \Phi(\lambda,s),
\end{equation}
which  relates unstable pole $s\in \overline{\mathbb C}_+$ of $G$ and damping coefficient $\lambda > 0$ to the pair
$(q,\alpha)$ through the operator $\Phi$. Since this is a complex equation and $q,\alpha$ are real, we deduce
\begin{equation}
\label{poles}
\alpha = -\frac{{\rm Im}\, \Phi(\lambda,s)}{{\rm Im}(s)}, \quad
q-1 =  \frac{ {\rm Re} \,\Phi(\lambda,s) {\rm Im}(s) - {\rm Im} \, \Phi(\lambda,s) {\rm Re}(s) }{{\rm Im}(s)}.
\end{equation}
We have proved the following
\begin{lemma}
Let $\lambda > 0$ and $s \in \overline{\mathbb C}_+$. Suppose $(q,\alpha)$ given by {\rm (\ref{poles})} is in $\mathbb R_+^2$. Then
$s$ is an unstable pole of $G$ for the parameters $(q,\alpha,\lambda)\in \mathbb R_+^3$.
\hfill $\square$
\end{lemma}

Let us look at poles on the imaginary axis
$j\mathbb R$, referred to as {\it zero-crossings}. Going back to (\ref{poles1}) with $s=j\omega$ gives
\begin{lemma}
Let $\lambda > 0$ and $\omega\in \mathbb R$. Suppose the pair
\[
q-1 = {\rm Re}\, \Phi(\lambda,j\omega), \quad \alpha = -\frac{{\rm Im}\,\Phi(\lambda,j\omega)}{\omega}
\]
satisfies $(q,\alpha)\in \mathbb R_+^2$. Then $j\omega$ is a zero crossing (unstable pole on $j\mathbb R$) of
$G$ for the parameter $(q,\alpha,\lambda)\in \mathbb R_+^3$.
\hfill $\square$
\end{lemma}
Let us look more specifically at zero-crossings through the origin. Substituting $s=0$ in the denominator $d(s)$ in  (\ref{TF}) 
and equating $d(0)=0$ 
gives
the relation
\[
q-1 = 2 \lambda_{\rm crit}
\]
which says that for $q > 1$ a real  pole of $G$ crosses the imaginary axis through the origin at the critical value $\lambda = \lambda_{\rm crit}$.
Here we use the fact that $\Phi(\lambda,0)=2\lambda$, explained by the relations  
\begin{equation}
    \textstyle\frac{e^\sigma-e^{-\sigma}}{2\sigma} = 1 + \frac{\sigma^2}{3!} + \frac{\sigma^4}{5!} + \dots,  \frac{e^\sigma+e^{-\sigma}}{2}
= 1 + \frac{\sigma^2}{2!} + \frac{\sigma^4}{4!} + \dots,   \frac{\sigma^2}{s} = s + 2\lambda.
\end{equation}

\begin{theorem}
\label{theorem1}
For fixed $(q,\alpha,\lambda)\in \mathbb R_+^3$ there exists $R > 0$ such that
$G$ has no poles and no transmission zeros on $\{s \in \overline{\mathbb C}_+: |s|  \geq  R\}$.
\end{theorem}

\begin{proof}
1)
For unstable poles
we have to show that equation (\ref{poles1}), respectively, (\ref{poles}) has no solutions when $s\in \overline{\mathbb C}_+$ and
$|s|\gg 0$ sufficiently large.
Let $s = \mu + j\omega$, $\sigma = a + jb$, then by the definition of $\sigma$:
\begin{equation}
\label{s-sig}
a^2-b^2 = \mu^2-\omega^2 + 2\lambda \mu, \quad ab = \omega(\mu + \lambda).
\end{equation}
It follows that for fixed $a_0 > 0$ the set
$\{s \in \overline{\mathbb C}_+: {\rm Re}(\sigma) \leq a_0\}$ is bounded. Choose $R_1 > 0$
such that $\{s \in \overline{\mathbb C}_+: {\rm Re}(\sigma) \leq a_0\} \subset \{s\in \overline{\mathbb C}_+: |s| \leq R_1\}$.
It remains to discuss
candidate poles $s\in \overline{\mathbb C}_+$ with Re$(\sigma) \geq a_0$ for some fixed $a_0 > 0$.

2)
Consider $s\in \overline{\mathbb C}_+$ with Re$(\sigma) =a\geq a_0$ and
define
\[
\theta :=
\frac{e^\sigma+e^{-\sigma}}{e^\sigma-e^{-\sigma}} = \frac{1 + e^{-2a} e^{-j 2b}}{1-e^{-2a} e^{-j 2b}}.
\]
Then $\frac{1+\rho_0}{1-\rho_0} \geq |\theta| \geq \frac{1-\rho_0}{1+\rho_0}$, where $\rho_0 = e^{-2a_0}$.
Moreover, we have $|\theta+1| \geq \frac{2}{1+e^{-2a_0}} =: \theta_0 > 1$. 
Now
choose $\epsilon > 0$ such that
$\frac{1+\rho_0}{1-\rho_0}\, \epsilon < \theta_0/2$. 
Since $\lambda$ is fixed we have $\sigma/s \to 1$ as $s\to \infty$ on $\mathbb C^+$, hence there exists $M=M(\lambda) >0$ such that
$|s-\sigma| < \epsilon |s|$ for all $|s| \geq M$. Then
$|s+\theta \sigma| = |\theta (\sigma-s) + (\theta+1) s| \geq |\theta+1| |s| - |\theta| |s-\sigma| \geq
\theta_0 |s| - |\theta| \epsilon |s| \geq \theta_0/2 |s|$ for $|s| \geq M$.
Writing (\ref{poles1}) as
\begin{equation}\label{sim1}
q-1-\alpha s = \Phi(\lambda ,s) = \frac{2\lambda}{s + \theta\sigma},
\end{equation}
and taking into account that on the right hand side we now have
\[
|\Phi(\lambda,s)| = \frac{2\lambda}{|s+\theta \sigma|} \leq \frac{4\lambda}{\theta_0|s|} 
\]
we see that (\ref{sim1}) can have no solution for $|s| \geq \max\{\frac{4\lambda}{\alpha \theta_0} , 1+\frac{1+q}{\alpha}, M,R_1 \}=:R$.
That settles the case  $\alpha > 0$.

3)
For $\alpha = 0$ and $q\not= 1$ there are no poles in $|s| > \frac{4\lambda}{|q-1|\theta_0}$, and for
$q=1$, $\alpha = 0$ clearly (\ref{sim1}) has no solutions.

4) Let us next discuss unstable zeros. Clearly those can only occur in the second component $y_2/u$ in 
(\ref{TF}). Here the equation is $\Phi(\lambda,s)^{-1} = \frac{(q+1)s - \alpha s^2}{2\lambda}$. From 1) above
we know that we may concentrate on Re$(\sigma) \geq a_0$, and from 2)  we have
$|\theta|\leq \frac{1+\rho_0}{1-\rho_0}$, while   for $|s| > \lambda$ we get $|\sigma| \leq \sqrt{3}|s|$, hence  for 
$|s| > \max\{R_1,\lambda\}$:
\[
|\Phi(\lambda,s)| = \left| \frac{2\lambda}{s+\theta \sigma}  \right| \geq \frac{2\lambda}{|s| + |\theta| |\sigma|} \geq
 \frac{2\lambda}{|s| (1+\frac{1+\rho_0}{1-\rho_0} \sqrt{3})}. 
\]
This leads to
\[
\frac{| (q+1)s - \alpha s^2 |}{2\lambda} = |\Phi(\lambda,s)|^{-1} \leq \frac{|s| (1+\frac{1+\rho_0}{1-\rho_0}\sqrt{3})}{2\lambda}
\]
hence
\[
\alpha |s| - (q+1) \leq 
|\alpha s - (q+1)| \leq 1+\frac{1+\rho_0}{1-\rho_0}\sqrt{3}.
\]
For $\alpha > 0$ this cannot be satisfied for large $|s|$. In fact,
there are no unstable zeros
on $|s| > R:= \max\{R_1,\lambda, (2+q+\frac{1+\rho_0}{1-\rho_0}\sqrt{3})/\alpha\}$. For $\alpha = 0$
the equation for unstable zeros is $\theta\sigma = qs$, and since $\sigma/s = \sqrt{1+2\lambda/s} \to 1$ for $s\to\infty$, we get
$\theta \to q$. On choosing $a_0$ sufficiently large, we get $\theta \approx 1$, which leads to a contradiction for $q\not=1$. 
Finally, for $q=1, \alpha=0$ we obtain the transfer function
$y_2/u=\frac{s[(\sigma-s)e^\sigma +(\sigma+s)e^{-\sigma}]}{(\sigma-s)(\sigma+s)e^\sigma- (\sigma+s)^2e^{-\sigma}}$,
so unstable zeros $\not=0$ satisfy $\frac{s-\sigma}{s+\sigma}=e^{-2\sigma}$. That gives
$-\frac{\lambda}{s+\sigma+\lambda} = e^{-2\sigma}$, which cannot be satisfied for large $|s|$.
\hfill $\square$
\end{proof}

Since the transfer function $G$ is of size $2 \times 1$, the number of unstable poles
is the maximum of the number of unstable poles of $G_1(s)=1/d(s)$ and $G_2(s)=n(s)/d(s)$, 
hence the number of unstable zeros of $d(s)$.  The latter can be determined by the argument principle. 
For the following we denote the half circle used for the standard Nyquist contour
by $\afunc{D}_R$.

\begin{proposition}
Suppose $(q,\alpha,\lambda)\in \mathbb R_+^3$ does not give rise to zero crossings. Then the number $n_p$ of unstable poles
of $G(s)$ equals the winding number of $d(\afunc{D}_R)$ around $0$, where the radius  $R>0$ is as in Theorem {\rm \ref{theorem1}}.
\end{proposition}

The radius $R$ in Theorem \ref{theorem1} may be quantified, and the winding number can be computed exactly
using the method in \cite{AN:18}. 
 If $(q,\alpha,\lambda)$ creates a zero-crossing, the contour
$\afunc{D}_R$ has to be modified, either by making small indentations into the right half plane, or preferably by removing poles on
$j\mathbb R$ with the method of \cite{chinese}, as explained in \cite{AN:18}.
%
%
%
%
%
%
%
 At this stage we have completed step 1 of our general algorithm \ref{algo1}.

We conclude this section with the following
important consequence of Theorem \ref{theorem1}.

\begin{corollary}
The input-output map and the input-to-state map of the boundary control problem {\rm (\ref{system})} are bounded.
\end{corollary}

\begin{proof}
As a consequence of \cite[Thm. 2.3]{morris} for input-output
boundedness it suffices to show that $\sup_{{\rm Re}(s)>a_0}|G_2(s)|< \infty$ for some 
$a_0\in \mathbb R$. 
We choose $a_0$ as in the proof of Theorem \ref{theorem1}, 
which allows to bring $\theta$ as close to 1 as we wish. Now
with the notation of the theorem
\[
G_2(s)=\frac{\sigma\theta+\alpha s^2-qs}{\sigma^2/s + \alpha s^2 -qs + 2\sigma\theta(\alpha s-q+1)}.
\]
For $\alpha > 0$ we divide numerator and denominator by
the leading term $\alpha s^2$, which gives
\[
G_2(s)=\frac{1 +\sigma\theta/\alpha s^2 - q/\alpha s}
{1+1/\alpha s + 2\lambda/\alpha s^2 -q/\alpha s + 2\theta \sigma/s + 2(1-q)\sigma\theta/\alpha s^2}
\sim\frac{1}{1+2\theta \sigma/s}.
\]
But $\sigma^2/s^2 = 1+2\lambda/s \sim 1$, whence
$G_2(s) \sim 1/3$, showing that $G_2$ is bounded
on some half plane Re$(s) > a_0$. Since $G_2=n/d$
and $G_1=1/d$, this is also true for $G_1$. In the case
$\alpha=0$ simplification by $s$ leads to a similar
estimate. 
\hfill $\square$
\end{proof}

\section{Pattern of unstable poles}
\label{pattern}
As a consequence of 
the previous section we can determine the number $n_p$ of unstable poles of $G$ for every scenario
$(q,\alpha,\lambda)\in \mathbb R_+^3$ using the argument principle. However, we would like to learn a little more about $n_p(q,\alpha,\lambda)$,
and in this section we shall see that $n_p\in \{0,1,2\}$, where the corresponding regions can be determined with arbitrary
numerical precision.

To begin with, observe that for  $\lambda = 0$
the transfer function (due to $\sigma = s$) simplifies to a pure delay system
\[
G_{\lambda=0}(s) = 
\begin{bmatrix}
\displaystyle
\frac{e^{-s}}{1+\alpha s-q} \\\\
\displaystyle
\frac{(1+\alpha s -q)+ (1-\alpha s +q)e^{-2s}}{2(1+\alpha s-q)}
\end{bmatrix}
=
\begin{bmatrix}
\displaystyle
\frac{\frac{1}{\alpha} e^{-s}}{s- \frac{q-1}{\alpha}}
\\\\
\displaystyle
\frac{1}{2} + \frac{\frac{1}{2}\left(  \frac{1+q}{\alpha} -s \right)e^{-2s}}{s-\frac{q-1}{\alpha}}
\end{bmatrix},
\]
where we immediately see that $G_{\lambda=0}$ has one unstable real pole if $q \geq 1$, while it is stable
for $q < 1$. 

This suggests now the following procedure. Fix $(q,\alpha) \in \mathbb R^2_+$,  and then follow
the evolution of the number of unstable poles $n_p(\lambda):=n_p(q,\alpha,\lambda)$ of $G$ as $\lambda$ increases from $\lambda = 0$ to $\lambda \to +\infty$. 
We know the number  of poles at $\lambda = 0$, and we expect that for very large $\lambda \gg 0$ the damping effect in the wave 
equation should lead back to stability, $n_p(\lambda \gg 0)=0$. 

Let us look again at zero crossings at the origin. We know that for $q > 1$ 
the origin is crossed when $\lambda\in [0,\infty)$ reaches the critical value $\lambda_{\rm crit} = (q-1)/2>0$.
We have to decide whether this real pole when crossing $s=0$ migrates from left to right or in the opposite direction.
Let
$s(\lambda)$ be the position of the potentially unstable pole on the real axis, that is
$d\left( s(\lambda),\lambda  \right)=0$, where $s(\lambda_{\rm crit}) = 0$. Differentiation with respect to $\lambda$ gives
\[
s'(\lambda) = -\frac{d_\lambda(s(\lambda),\lambda)}{d_s(s(\lambda),\lambda)},
\]
where
\[
\frac{\partial d}{\partial \lambda}=
2\left( 1+\frac{\sigma^2}{3!} + \dots \right) + (s+2\lambda+\alpha s^2-qs) \frac{s}{\sigma} \left( \frac{2\sigma}{3!} + \frac{4\sigma^2}{5!} + \dots \right)
+ (\alpha s-q+1) \frac{s}{\sigma} \frac{e^\sigma+e^{-\sigma}}{2}
\]
and
\begin{align*}
\frac{\partial d}{\partial s}&=(1+2\alpha s-q)\left( 1+\frac{\sigma^2}{3!}+\dots  \right) + (s+2\lambda+\alpha s^2-qs) \frac{s+\lambda}{\sigma}
\left( \frac{2\sigma}{3!}+\frac{4\sigma^3}{5!}+\dots  \right) \\
&
\qquad + \alpha \frac{e^\sigma+e^{-\sigma}}{2} + (\alpha s-q+1) \frac{e^\sigma-e^{-\sigma}}{2\sigma}(s+\lambda).
\end{align*}
Substituting $\lambda = \lambda_{\rm crit}=(q-1)/2$ and $s =s(\lambda_{\rm crit}) = 0$ gives
\[
s'(\lambda_{\rm crit}) = \frac{2}{\frac{1}{3} (q-1)^2 + (q-1) -\alpha}.
\]
Hence
\[
s'(\lambda_{\rm crit} )\left\{
\begin{matrix}
> 0 & \mbox{ for } \alpha < \frac{1}{3} (q-1)^2 + (q-1) \\
<0 & \mbox{ for } \alpha > \frac{1}{3}(q-1)^2 + (q-1)
\end{matrix}
\right.
\]
This leads to the following
\begin{lemma}
Let $(q,\alpha) \in \mathbb R_+^2$. If $\alpha < \frac{1}{3}(q-1)^2 + (q-1)$, then a single real pole of $G$ crosses the imaginary
axis through the origin at $\lambda = \lambda_{\rm crit} = (q-1)/2$ from left to right, going from stable at $\lambda_{\rm crit}-0$ to unstable at $\lambda_{\rm crit} + 0$.
If $\alpha > \frac{1}{3}(q-1)^2 + (q-1)$ a single real pole crosses the imaginary axis through the origin from right to left, going
from unstable at $\lambda = \lambda_{\rm crit} -0$ to stable at $\lambda = \lambda_{\rm crit} +0$.
\hfill $\square$
\end{lemma}
This can also be corroborated by investigating the value $G(0)$ in (\ref{TF}). We have
\[
G_1(0) = \frac{1}{2\lambda-q+1},
\]
so $G$ has no unstable pole at the origin, except for the critical $\lambda$ value $\lambda_{\rm crit} = (q-1)/2$ when $q > 1$. 
On the exceptional manifold $\mathbb M = \{(q,\alpha,\lambda)\in \mathbb R_+^3: 2\lambda = q-1\}$, we have
\[
\lim_{s\to 0} s G_1(s) = \frac{1}{\alpha - (q-1) - \frac{1}{3}(q-1)^2}, 
\]
which means a pole of order one at the origin, except when 
$(q,\alpha)$ lies on the parabola $\alpha = (q-1) + \frac{1}{3}(q-1)^2$. On the exceptional set
$\mathbb O=\{(q,\alpha,\lambda)\in \mathbb R_+^3: 2\lambda = q-1, \alpha = (q-1) + \frac{1}{3}(q-1)^2\}$ we find that
\[
\lim_{s\to 0} s^2 G_1(s) = \frac{6}{q^2-1},
\]
which means $G$ has a double pole at the origin, except when
$q = 1$.  The case $q=1$ now leaves only the parameter choice $(q,\alpha,\lambda)=(1,0,0)$, 
an exceptional point where the system is not well-posed.

Using the mapping $\Phi$, one can see
that the positive quadrant $(q,\alpha) \in \mathbb R_+^2$ may be divided into 5 different zones, shown in Fig. \ref{zones},
in which the number of unstable poles of $G$ evolves differently. Each zone has its specific pattern.

The red zone is ${\tt Red} = \{(q,\alpha): q \geq 1, \alpha \geq 0, \alpha \leq \frac{1}{3}(q-1)^2 + (q-1)\}$ is below a parabola. Setting
\[
m(q) = \sup\{\alpha > 0: q-1= {\rm Re}\,  \Phi(\lambda, j\omega), \alpha=-\omega^{-1}{\rm Im}\,\Phi(\lambda,j\omega)
\mbox{ for certain $\omega > 0$, $\lambda > 0$} \},
\]
the magenta zone is defined as
\[
{\tt Mag} = \{(\alpha,q): q\geq 1,  \textstyle\frac{1}{3}(q-1)^2 + (q-1)\leq \alpha \leq m(q)\}
\]
delimited by the parabola and the analytic curve $\alpha=m(q)$. 
The green zone is
\[
{\tt Green} = \{(q,\alpha): q \geq 1, \alpha \geq m(q)\},
\]
where the curve $\alpha = m(q)$ separates magenta and green.
Finally, on setting
\[
b(\alpha) = \inf\{q:  q-1= {\rm Re}\,  \Phi(\lambda, j\omega), \alpha=-\omega^{-1}{\rm Im}\,\Phi(\lambda,j\omega)
\mbox{ for certain $\omega > 0$, $\lambda > 0$} \},
\]
the blue zone is ${\tt Blue} = \{(q,\alpha): \alpha \geq 0, b(\alpha)\leq q \leq 1\}$, which is the only bounded one. The boundary of the blue zone
described by the curve $q = b(\alpha)$ is just a different local parametrization of the same
analytic curve $\alpha = m(q)$ separating magenta and green.
This curve disappears into $\alpha < 0$ at $(1,0)$, where it is no longer of interest.
The gray zone ${\tt Gray}$ is what is left over from the strip
$0 \leq q \leq 1$, $\alpha \geq 0$ when removing the blue zone.

\begin{figure}[!ht]
\includegraphics[scale=1.0]{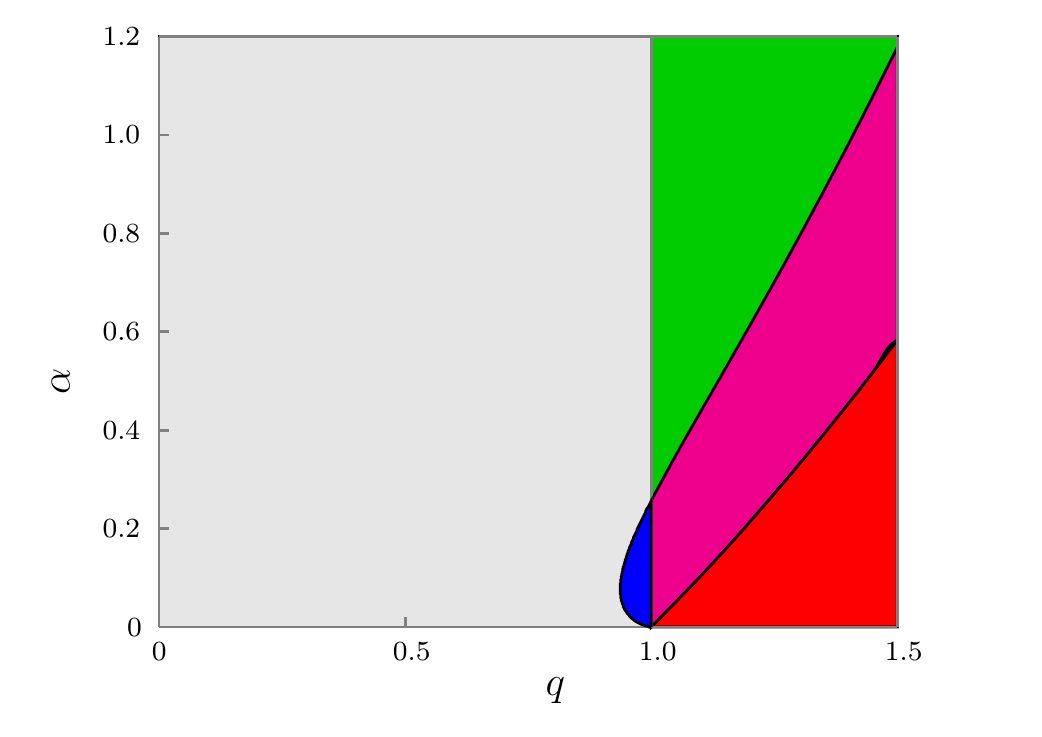}
\caption{Five regions and five scenarios.
The pattern of the gray zone is 0. Blue: 0-2-0. Red: 1-2-0. Magenta: 1-0-2-0. Green: 1-0.
\label{zones}}
\end{figure}

Altogether, we have found the following classification or pattern.
\begin{itemize}
\item For $(q,\alpha) \in {\tt Gray}$ the system $G$ is stable for all $\lambda \geq 0$. The pattern is ${\tt 0}$.
\item For $(q,\alpha) \in {\tt Blue}$ there exist $0 < \lambda_1(q,\alpha) < \lambda_2(q,\alpha)$ such that $G$
is stable for all $0\leq \lambda < \lambda_1(q,\alpha)$ and $\lambda > \lambda_2(q,\alpha)$, and has two unstable
poles for $\lambda_1 \leq \lambda \leq \lambda_2$. The pattern is {\tt 0-2-0}.
\item For $(q,\alpha)\in {\tt Red}$ the system has one unstable pole for $0 \leq \lambda \leq (q-1)/2=: \lambda_1(q)$, and two unstable poles
for 
$(q-1)/2 \leq \lambda \leq \lambda_2(q,\alpha)$,  while it is again stable for
$\lambda > \lambda_2(q,\alpha)$. The pattern is {\tt 1-2-0}.
\item For $(q,\alpha)\in {\tt Mag}$ there exist $\lambda_2(q,\alpha) > \lambda_1(q,\alpha) > (q-1)/2$ such that
the system has one unstable pole for $0 \leq \lambda \leq (q-1)/2$, no unstable poles
for $(q-1)/2 < \lambda < \lambda_1(q,\alpha)$, then  two unstable poles
for $\lambda_1(q,\alpha) \leq \lambda \leq \lambda_2(q,\alpha)$, and again no unstable poles  for $\lambda > \lambda_2(q,\alpha)$.
The pattern is {\tt 1-0-2-0}.
\item For $(q,\alpha)\in {\tt Green}$ the system has one unstable pole for $0 \leq \lambda \leq (q-1)/2$, and is stable for
$\lambda > (q-1)/2$. The pattern is {\tt 1-0}.
\end{itemize}

\section{Stabilization}
\label{stabilize}
In this section we construct finite-dimensional output feedback controllers which stabilize the linearization
$G$ of system (\ref{system})--(\ref{outputs}) exponentially. We 
start with the following

\begin{theorem}
\label{theorem2}
Let $K$ be a finite-dimensional output feedback controller for {\rm (\ref{system})-(\ref{outputs})}
which stabilizes the closed loop  in the $H_\infty$-sense. Then the closed loop is even exponentially stable.
\end{theorem}

\begin{proof}
1) Suppose the boundary control problem (BCP) is written in the abstract form
\[
\dot{x} = \mathscr Ax, \quad \mathscr Px = u, \quad y = \mathscr Cx
\]
with suitable unbounded operators \cite{morris,Salamon:87,Salamon:89}, 
and let
the controller $u=Ky$ stabilize BCP in the $H_\infty$-sense. Writing $K(s) = K_1(s) + K_0$ with $K_1$
strictly proper, we see that $\widetilde{u}=K_1y$ stabilizes the modified
BCP
\[
\dot{x} = \mathscr Ax, \quad (\mathscr P - K_0 \mathscr C)x=\widetilde{u},
\quad y = \mathscr Cx
\]
in the $H_\infty$-sense, 
where $\widetilde{u}=u-K_0y$. We will use this type
of shift to arrange for a strictly proper stabilizing 
controller.

2)
We start from (\ref{system})  by performing the
change of variables  $z(\xi,t) = x_\xi(\xi,t)$, $v(t) = x_t(0,t)$, cf. \cite{sagert}, which leads to
an equivalent representation of (\ref{system}) as a PDE coupled with and  ODE:
\begin{align}
\label{equivalent1}
G:\qquad
\begin{split}
z_{tt}(\xi,t) =&z_{\xi\xi}(\xi,t) -2\lambda z_t(\xi,t)  \\
z(1,t) &= \widetilde{u}(t) \\
\alpha z_\xi(0,t) &= z(0,t) + (q+2\alpha\lambda) v(t) \\
\alpha \dot{v}(t) &= z(0,t) + qv(t)
\end{split}
\end{align}
where the new state is $(z,z_t,v)$, the measured outputs are 
$$y_1=v, \quad y_2(t) = \int_0^1 z_t(\xi,t) \, d\xi+v(t),$$ 
and
where a new control $\widetilde{u}(t) = u(t) - x_t(1,t)= u(t) - y_2(t)$ is used. Since the controller $u=Ky$
stabilizes (\ref{system}) in the $H_\infty$ sense by hypothesis, 
so does $\widetilde{u}=Ky-y_2$ 
for (\ref{equivalent1}), and since the state trajectories remain unaffected,
we may from here on prove the statement
for controller $\widetilde{u}=Ky$ and system (\ref{equivalent1}).
It is also clear that we may replace the outputs $y_1,y_2$ by equivalent
outputs
$\widetilde{y}_1=v$, $\widetilde{y}_2=\int_0^1 z_t(\xi,t)d\xi$, because
$\widetilde{y}_1=y_1$, $\widetilde{y}_2=y_2-y_1$. Then $\widetilde{u}=u-y_2=u+\widetilde{y}_1-\widetilde{y}_2$,
and the controller is $\widetilde{u}=\widetilde{K}\widetilde{y}$.
At this stage for the ease of presentation we
drop the tilde notation and write the new control
and measurements again as $u$ and $y$.

3) Let the controller $K$
have the form $u(s) = K(s)y(s) =K_1(s)y + K_0y$
with direct transmission $K_0y = k_1y_1+k_2y_2$ and strictly
proper part $K_1(s)$.
We now apply the
idea of part 1)  and shift its direct transmission into the
plant. This leads to
\begin{align}
\label{sp}
G': \qquad
\begin{split}
    z_{tt} - z_{\xi\xi} +2\lambda z_t &= 0 \\
    z(1,t) &= u - k_1y_1 - k_2 y_2\\
    \alpha z_\xi(0,t) - z(0,t) &= (q+2\alpha \lambda) v(t)\\
    \alpha\dot{v} &= qv + z(0,t)
    \end{split}
\end{align}
with the outputs as before,  now in feedback with
$u=K_1(s)y$. Note that
$K_1$ still stabilizes 
(\ref{sp}) in the $H_\infty$-sense, and since the state trajectories remain the same we may 
prove exponential stability of the loop for
this pair $G',K_1$. Since $K_1$ is strictly proper,
the controller $\dot{u}=K_1y$, respectively,
$u = \frac{1}{s}K_1(s)y=K'(s)y$,  is proper and
may be represented in state space as 
\begin{align*}
K': \qquad \begin{split}
    \dot{x}_K &= A_Kx_K + B_K^1y_1+B_K^2y_2\\
    \dot{u} &= \widetilde{u}\\
    \widetilde{u} &= C_Kx_K + d_K^1y_1 + d_K^2y_2.
    \end{split}
\end{align*}
If the original state-space realization is
$K=\left[\begin{array}{c|c}a&b\\ \hline c&d\end{array}\right]$,
then $K_1=\left[\begin{array}{c|c} a&b\\ \hline c&0 \end{array} \right]$, and
$K'=\left[\begin{array}{c|c}  a&b\\ \hline ca & cb \end{array}  \right]=: \left[\begin{array}{c|c} A_K&B_K\\ \hline C_K&D_K \end{array}   \right]$. Since $H_\infty$-stability
of the loop is not altered by these transformations, we may
prove the statement for the pair $G',K'$.

4)
We now perform a less standard manipulation, which consists in transferring parts of the
system dynamics 
(\ref{sp}) into a new augmented controller $\widetilde{K}$.
We introduce a new artificial output ${y}_3= z(0,t)$   
in (\ref{sp}),   and consider the boundary wave equation
\begin{align}
\label{equivalent2}
\widetilde{G}:\qquad
\begin{split}
z_{tt} - z_{\xi\xi} + 2 \lambda z_t &=0\\
z(1,t) &={u}(t) -k_1v-k_2{y}_2 \\
\alpha z_\xi(0,t) - z(0,t) &= (q+2\alpha \lambda) v(t) \\
{y}_2(t) &= \int_0^1 z_t(\xi,t)d\xi \\
{y}_3(t) &= z(0,t)
\end{split}
\end{align}
Here we have substituted $v=y_1$, created a new input into $\widetilde{G}$, and have now an infinite
dimensional system $\widetilde{G}$
in feedback with the extended controller 
\begin{align}
\label{extended}
    \widetilde{K}: \quad
    \begin{split}
        \alpha \dot{v}&= \hspace{2cm} q v \qquad\qquad + {y}_3\\
        \dot{x}_K&= A_Kx_K + B_K^1v  + B_K^2{y}_2 \\
        \dot{{u}} &= C_K x_K \,+ d_K^1 v   + \,d_K^2 {y}_2\\
       \dot{v} &= \hspace{2cm}\frac{q}{\alpha} v \qquad\quad + \frac{1}{\alpha}y_3
    \end{split}
\end{align}
The ODE $\alpha \dot{v}=z(0,t)+qv = qv+{y}_3$ was
shifted from
$G'$ into the new  $\widetilde{K}$, leaving us with a simpler
infinite-dimensional system $\widetilde{G}$. 
The controller $\widetilde{K}$ is $K'$ augmented
by this ODE, so is still finite dimensional, and
moreover, is also an integral controller 
with regard to its new output $v$. 
The output ${y}_1$ has disappeared from
(\ref{equivalent2}), because the corresponding dynamics are now integrated in $\widetilde{K}$. The state of (\ref{equivalent2}) is 
$(z,z_t)$, while the state of $\widetilde{K}$ is $(v,x_K)$,
to which we have to add the integrator. 
$\widetilde{K}$ is an integral controller with regard to the new output $\dot{{v}}$ from (\ref{extended}).

5) 
Our next step is to find a state-space representation of ${y}=\widetilde{G}[{u};{v}]^T$ in (\ref{equivalent2}), 
which means representing it as a well-posed boundary
control system in the sense of 
\cite{Salamon:87,Salamon:89}, 
\cite[Def. 5.2.1]{Staffans_book} or \cite{morris}.
With zero boundary conditions equation (\ref{equivalent2})  reads
\begin{align*}
    z_{tt}-z_{\xi\xi} + 2\lambda z_t &=0\\
    z(1,t)+k_2 \int_0^1 z_t(\xi,t)d\xi &= 0\\
    \alpha z_\xi(0,t)-z(0,t)&= 0\\
    z(\xi,0)=z_0(\xi), &\;\;z_t(\xi,0)=z_1(t).
\end{align*}
This has now a representation as a strongly continuous semi-group
\begin{equation}
    \label{same}
\dot{\mathfrak{z}} = \begin{bmatrix} 0 & I \\ \frac{d^2}{d\xi^2} & -2\lambda \end{bmatrix} \mathfrak{z} =: A\mathfrak{z}, \quad \mathfrak{z}(0)=\mathfrak{z}_0,
\end{equation}
where $\mathfrak{z}=(z,z_t)$, and where the generator $A$ has
$D(A)=\{(z_1,z_2)\in H^2(0,1) \times H^1(0,1): z_1(1)+k_2\int_0^1 z_2(\xi,t)d\xi=0, \alpha z_{1x}(0)-z_1(0)=0\}$ as domain in the Hilbert space $H = H^1(0,1) \times L^2(0,1)$. 
Define $\mathscr A$ with domain $D(\mathscr A)=H^2\times H^1$ by the same formula (\ref{same}), and let the projector
$\mathscr P$ with $D(\mathscr P)=D(\mathscr A)$ be defined as $\mathscr P \mathfrak{z} = \begin{bmatrix} z(1)+k_2\int_0^1z_t(\xi)d\xi\\
\alpha z_\xi(0) -z(0)
\end{bmatrix} \in \mathbb C^2$.
The boundary control of $\widetilde{G}$ has  now the abstract form
\[
\dot{\mathfrak{z}} = \mathscr A \mathfrak{z}, \quad \mathscr P\mathfrak{z} 
= \begin{bmatrix} u -k_1 v\\ k_3v\end{bmatrix}
,\quad \mathfrak{y} = \mathscr C \mathfrak{z},
\]
where $k_3:=q+2\alpha\lambda$, and
where $\mathfrak{y}=[{y}_2 , {y}_3]$, with $\mathscr C:H^1 \times L^2 \to \mathbb C^2$ bounded.
Finally we re-arrange the boundary condition by defining $\mathscr P'$ with $D(\mathscr P')=D(\mathscr P)$
as
$$\mathscr P'\mathfrak{z}=\begin{bmatrix} z(1)+k_2\int_0^1 z_t(\xi)d\xi +\alpha \frac{k_1}{k_3}z_\xi(0)-\frac{k_1}{k_3}z(0)\\
\frac{1}{k_3}(\alpha z_\xi(0)-z(0))\end{bmatrix},\quad
\mathscr P'\mathfrak{z} 
=\begin{bmatrix}u\\v\end{bmatrix}=:\mathfrak{u}.$$

In order to make this well defined, we have according to \cite[Sect. 3.3]{CurtainZwart:1995} 
to assure that $D(\mathscr A)\subset D(\mathscr P')$, $D(A) = D(\mathscr A) \cap {\rm ker}(\mathscr P')$, $A\mathfrak{z}=\mathscr A\mathfrak{z}$
on $D(A)$, and that $A$ generates a $C_0$-semi group. These are satisfied by construction. In addition, we require an
operator $B:\mathbb C^2\to H$ such that $\mathscr P' \circ B = I$, im$(B)\subset D(\mathscr A)$ and $\mathscr A \circ B$ bounded. This can be defined
by the ansatz
$$B\mathfrak{u} = \begin{bmatrix} b(\xi){u} + c(\xi) {v} \\ 0\end{bmatrix}, \quad b(\xi) = \textstyle \xi^2, \quad c(\xi)=
(k_3-k_1)\xi^2-k_3.$$
Indeed, $\mathscr P' B\mathfrak{u}= \begin{bmatrix} b(1)u+c(1)v+d_K^1(\alpha b'(0)u+\alpha c'(0)v-b(0)u-c(0)v) \\ \alpha b'(0)u+\alpha c'(0)v-b(0)u-c(0)v  \end{bmatrix}
=\begin{bmatrix} {u}\\{v}\end{bmatrix} = \mathfrak{u}$. As is well-known, (cf. \cite[Sect. 3.3]{CurtainZwart:1995}), 
the boundary wave equation may now be represented by the
state-space
\begin{equation}
    \label{equivalent3}
\dot{\mathfrak{x}} = A\mathfrak{x} - B \dot{\mathfrak{u}} +\mathscr A B\mathfrak{u}, \quad \mathfrak{x}(0)=\mathfrak{x}_0,
\end{equation}
where solutions $\mathfrak{z}$ of (\ref{equivalent2}) and $\mathfrak{x}$ of (\ref{equivalent3}) are related by $\mathfrak{x}(t)= \mathfrak{z}(t)-B\mathfrak{u}(t)$.
This can be further streamlined as
\begin{equation}
    \label{equivalent4}
\dot{\mathfrak{x}^e} = \begin{bmatrix} 0&0\\ \mathscr A B&A\end{bmatrix}\mathfrak{x}^e + \begin{bmatrix}I\\ -B \end{bmatrix} \widetilde{\mathfrak{u}},
\end{equation}
where the extended state is $\mathfrak{x}^e=(\mathfrak{u},\mathfrak{x})=({u},{v},z,z_t)$, 
and where $\widetilde{\mathfrak{u}}=\dot{\mathfrak{u}}$ has become the input.
The output operator for (\ref{equivalent4}) is now $\mathfrak{y} = \mathscr C^e \mathfrak{x}^e = \mathscr C \circ [B \;\; I]\begin{bmatrix}u\\\mathfrak{z}-Bu\end{bmatrix}=\mathscr C \mathfrak{z}$.

6) We next show that system $\widetilde{G}$, and therefore also the state-space representation (\ref{equivalent4}) with $C_0$-semi group,
is exponentially stabilizable. This can for instance be obtained from  \cite[Theorem 4.2]{sagert}, where the authors construct a state feedback controller 
which stabilizes
(\ref{equivalent1})  exponentially in the Hilbert space $H=H^1\times H^2 \times L^2$. The control law found in that reference can be arranged as a state feedback law for $\widetilde{G}$, and hence for (\ref{equivalent4}), using the same technique of shifting parts of the dynamics from
plant to controller. Alternatively, we may even use the open loop characterization of stabilizability, called optimizability in
\cite{weiss_rebarber:98}, which is equivalent to stabilizability, while offering a more convenient way to check it.

7) We now show that the controller $\widetilde{K}$
is admissible for $\widetilde{G}$ and is as a system
exponentially stabilizable. Due to shifting the direct transmission of $K$ into the plant as outlined in 1)
and put to work in (\ref{equivalent2}), the new controller
$\widetilde{K}$ in (\ref{extended}) is 
written as an integral controller, that is, its output
is $\widetilde{\mathfrak{u}}=\dot{\mathfrak{u}}=[\dot{u},\dot{v}]$, which makes it
admissible for $\widetilde{G}$.

Assuming that the original $K=\left[\begin{array}{c|c} a&b\\\hline c&d \end{array} \right]$ is stabilizable and detectable (e.g. minimal), the same is true for
$K'$ obtained in 3), so $\left[\begin{array}{c|c} A_K&B_K\\\hline C_K&D_K \end{array} \right]$ 
is stabilizable. Now the augmented controller is
$
\widetilde{K}=\left[\begin{array}{c|c} \widetilde{A}_K&\widetilde{B}_K\\\hline \widetilde{C}_K&\widetilde{D}_K \end{array} \right]$ 
with
$\widetilde{A}_K=\begin{bmatrix}A_K&B_K^1\\0&\alpha^{-1}
\end{bmatrix}$, $\widetilde{B}_K=\begin{bmatrix} B_K^2&0\\0&\alpha^{-1} \end{bmatrix}$,
$\widetilde{C}_K=\begin{bmatrix}C_K&d_K^1\\0&\alpha^{-1}
\end{bmatrix}$, $\widetilde{D}_k=\begin{bmatrix}d_K^2&0\\0&\alpha^{-1}\end{bmatrix}$. Applying the Hautus test, 
for simplicity in the case $\lambda \not=\alpha^{-1}$,
let
$v$ be an eigenvector of $A_K^T$ with unstable eigenvalue $\lambda$, then $[v\; \rho]^T$ is an eigenvector of 
$\widetilde{A}_K^T$ for $\lambda$ if $\rho=B_K^{1T}v/(\lambda-\alpha^{-1})$. Now
$\widetilde{B}_K^T [v\; \rho]^T= [v \; v(\alpha^{-1}/(\lambda-\alpha^{-1}))]^T$, and this vector
cannot be $=[0 \;0]^T$, because that would imply 
$B_K^{1T}v=0$ and $B_K^{2T}v=0$, hence $B_K^Tv=0$,
contradicting stabilizability of $[A_K,B_K,C_K,D_K]$.
Now for the eigenvalue $\alpha^{-1}$ of 
$\widetilde{A}_K^T$ we take the eigenvector
$w=[0\; 1]^T$, then $\widetilde{B}_K^Tw=[0\; \alpha^{-1}]^T
\not=[0\; 0]^T$, which proves stabilizability.

With $\widetilde{G}$ and $\widetilde{K}$ exponentially stabilizable, the closed loop $(\widetilde{G},\widetilde{K})$ is also exponentially
stabilizable (see \cite[Prop. 8.2.10(ii)(c)]{Staffans_book})  in the sense of the induced state-space realization \cite[Chap. 7]{Staffans_book}. 
The infinitesimal generator
of the closed loop will be denoted as $A_{cl}$.

8)
Next we argue that  $\widetilde{G}$ is exponentially detectable. Since $\widetilde{G}$ is exponentially stabilizable,
its semi-group satisfies the spectrum decomposition assumption,  see \cite[Theorem 5.2.6]{CurtainZwart:1995}. Since from the discussion of section
\ref{linear} we know that there are only finitely many 
right hand poles, all with finite multiplicity, a necessary and sufficient condition for 
exponential detectability is that ker$(sI-\widetilde{A}) \cap {\rm ker}(\widetilde{C}) = \{0\}$ for every $s \in \overline{\mathbb C}_+$; 
see \cite[Theorem 5.2.11]{CurtainZwart:1995}, where $(\widetilde{A},\widetilde{B},\widetilde{C})$  
refers to the state-space realization of $\widetilde{G}$ derived in 2) above.
But  that may now be checked in the frequency domain. It means that for every $s\in \overline{\mathbb C}_+$ the only solution of the Laplace transformed
system (\ref{laplace})  with $u=0$ satisfying $y_1(s)=sx(0,s)=0$, $y_2(s)=sx(1,s)=0$ is $x\equiv 0$. Now for $s \not=0$ these boundary conditions
give $x(0,s)=0$, $x(1,s)=0$, and therefore from the boundary conditions in (\ref{laplace})  $x_\xi(0,s)=0$, $x_\xi(1,s)=0$. The general solution
of the dynamic equation in (\ref{laplace}) being $x(\xi,s) =k_1e^{\sigma \xi} + k_2e^{-\sigma \xi}$, with constants depending on $s$, we get the four conditions
$k_1+k_2=0$, $\sigma(k_1-k_2)=0$, $k_1e^\sigma + k_2e^{-\sigma}=0$, $\sigma(k_1e^\sigma - k_2 e^{-\sigma})=0$, which can only be satisfied if $k_1=k_2=0$.

With $\widetilde{G}$ exponentially detectable, and $\widetilde{K}$ exponentially detectable with an argument 
similar to 7) above, the closed
loop is exponentially detectable, again by \cite[Prop. 8.2.10(ii)(c)]{Staffans_book}.

9) According to  \cite[Theorem 5.2]{Morris:1999}, a well-posed system which is exponentially stabilizable,
exponentially detectable, and at the same time $H_\infty$-stable, is already exponentially stable in the
state-space sense, i.e., the generator of its semi-group 
is exponentially stable. We apply this to the closed loop system with generator $A_{cl}$. For this result
see also \cite[Thm. 1.1]{weiss_rebarber:98}, and \cite[8.35]{engel_nagel} for a classical antecedent.
\hfill $\square$
\end{proof}

\begin{corollary}
Let $K$ be a finite-dimensional controller for  {\rm (\ref{system})}-{\rm(\ref{outputs})}, and suppose the closed loop
with $G$ has no
unstable poles. Then $K$ stabilizes $G$ exponentially and $G_{\rm nl}$  locally
exponentially.
\end{corollary}

\begin{proof}
The result follows from Theorem \ref{theorem2} above once we show that $\widetilde{K}$ in (\ref{extended})
stabilizes $\widetilde{G}$ in (\ref{equivalent1}) in the $H_\infty$-sense. Since the transfer functions
are not concerned by the transformations in the proof of theorem \ref{theorem2}, it suffices to 
show that $K$ stabilizes $G$ in the $H_\infty$-sense. For that we have to show that
the closed loop transfer operator
\[
T(s) = \begin{bmatrix} I&G(s)\\-K(s)&I \end{bmatrix}^{-1} =
\begin{bmatrix}
(I+KG)^{-1} & -K(I+GK)^{-1}  \\ (I+GK)^{-1} G & (I+GK)^{-1}
\end{bmatrix}
\in {\bf H}_\infty
\]
belongs to the Hardy space ${\bf H}_\infty$. Since we know by hypothesis
that $T(s)$ has no poles in $\overline{\mathbb C}_+$, this follows as soon as $T$ is bounded on $j\mathbb R$.
Since $K$ is proper, this hinges on the behavior of $G$ on $j\mathbb R$. As is easy to see, the denominator
$d(s)$ of (\ref{TF}) satisfies $\lim_{\omega \to \infty} |d(j\omega)| = \infty$, so $y_1(s)/u(s)$ is proper, and it remains to show that
$y_2(s)/u(s)$ in (\ref{TF}) is bounded on $j\mathbb R$. Dividing numerator and denominator of $n_2/d$ in (\ref{laplace}) by $(e^\sigma-e^{-\sigma})/2\sigma$,
and observing that the term $\frac{(e^\sigma + e^{-\sigma})/2}{(e^\sigma-e^{-\sigma})/2\sigma}$ is bounded
on $j\mathbb R$, we see that $y_2(s)/u(s)$ is bounded, because the leading terms in both numerator and denominator
are now $\alpha s^2 = -\alpha \omega^2$.
That proves $H_\infty$-stability of the closed loop hence exponential stability of the linear closed loop.

It remains to show that $K$ stabilizes $G_{\rm nl}$ locally exponentially.  Due to the specific form of the non-linearity, this may be obtained with
\cite{zwart}.
\hfill $\square$
\end{proof}

\begin{remark}
Semi-groups for hyperbolic equations with boundary dynamics have been investigated, e.g. in \cite{mugnolo}, but as this 
requires additional conditions, we believe that our method of simplifying the infinite-dimensional part
by augmenting the controller offers additional flexibility. 
After all
the goal is to show that the closed-loop has an
exponentially stabilizable and detectable semi-group $e^{A_{cl}t}$, not necessarily the individual parts.
\end{remark}

We have achieved that for any finite-dimensional controller $K$ for (\ref{system})-(\ref{outputs}), 
exponential stability of the infinite dimensional closed loop with $G$ can be
verified by the Nyquist test. What remains to do is actually {\it find} such
a stabilizing controller.
A straightforward idea is to use a discretization of (\ref{system}), the most obvious being finite differences
\begin{align*}
x_i(t) &= x(\xi_i,t), \xi_i = ih, i=0,\dots,N, Nh=1 \\x_\xi(\xi_i,t) &\approx \frac{x_{i+1}(t)-x_{i-1}(t)}{2h}, \; x_{\xi\xi}(\xi_i,t)\approx \frac{x_{i+1}(t) +x_{i-1}(t)-2x_i(t)}{h^2}.
\end{align*}
With the boundary condition at $\xi=0$
\[
\alpha x_0''(t) = \frac{x_1(t)-x_{-1}(t)}{2h} + qx_0'(t)
\]
we can eliminate $x_{-1}$, and with the boundary condition at $\xi=1$
\[
\frac{x_{N+1}(t)-x_{N-1}(t)}{2h} = -x_N'(t) + u(t)
\]
we  eliminate $x_{N+1}$. Putting $\tilde{x}_i = x_i'$, $i=0,\dots,N$, we get a dynamical system of order $2N+2$
\begin{equation}
\label{tridiag}
\begin{bmatrix} x'\\\tilde{x}'\end{bmatrix} = \begin{bmatrix} 0 & I \\ T & \Lambda\end{bmatrix} \begin{bmatrix} x\\\tilde{x}\end{bmatrix} 
+ \begin{bmatrix} 0 \\b \end{bmatrix} u, \quad y_1(t) = \tilde{x}_0(t), y_2(t) = \tilde{x}_N(t),
\end{equation}
with typical $A$-matrix featuring a tridiagonal $T$ and a diagonal  $\Lambda$. It comes as a mild surprise that (\ref{tridiag})
is not stabilizable, the reason being a pole/zero cancellation at the origin. 

Kalman reduction using the function {\tt minreal} from
\cite{RCT2013b} removes one state of (\ref{tridiag}) and furnishes a stabilizable system, which we use for synthesis, and
where the reduced system $A$-matrix is now no longer sparse. 
In our experiment we chose $N=50$ and synthesized
controllers of various simple structures like a sum of PIDs $u = {\tt PID}_1y_1 + {\tt PID}_2y_2$,   a 5th-order state-space controllers, or on ignoring one of the outputs, standard PID controllers  $u = {\tt PID}\,y_1$, respectively, $u = {\tt PID}\,y_2$. 
These controllers, once they stabilize the reduced finite-dimensional system, 
are then tested against the infinite dimensional system using the Nyquist test of \cite{AN:18}, which gives
an exact answer. For instance,
stabilizing  the gray and blue scenarios with the 5th-order controller given in (\ref{controllers})
leads to the Nyquist plots in Fig. \ref{nyquist_blue} for the gray and blue scenarios, and certifies infinite-dimensional stability. 
\begin{figure}[!ht]
\begin{center}
\includegraphics[scale=0.4]{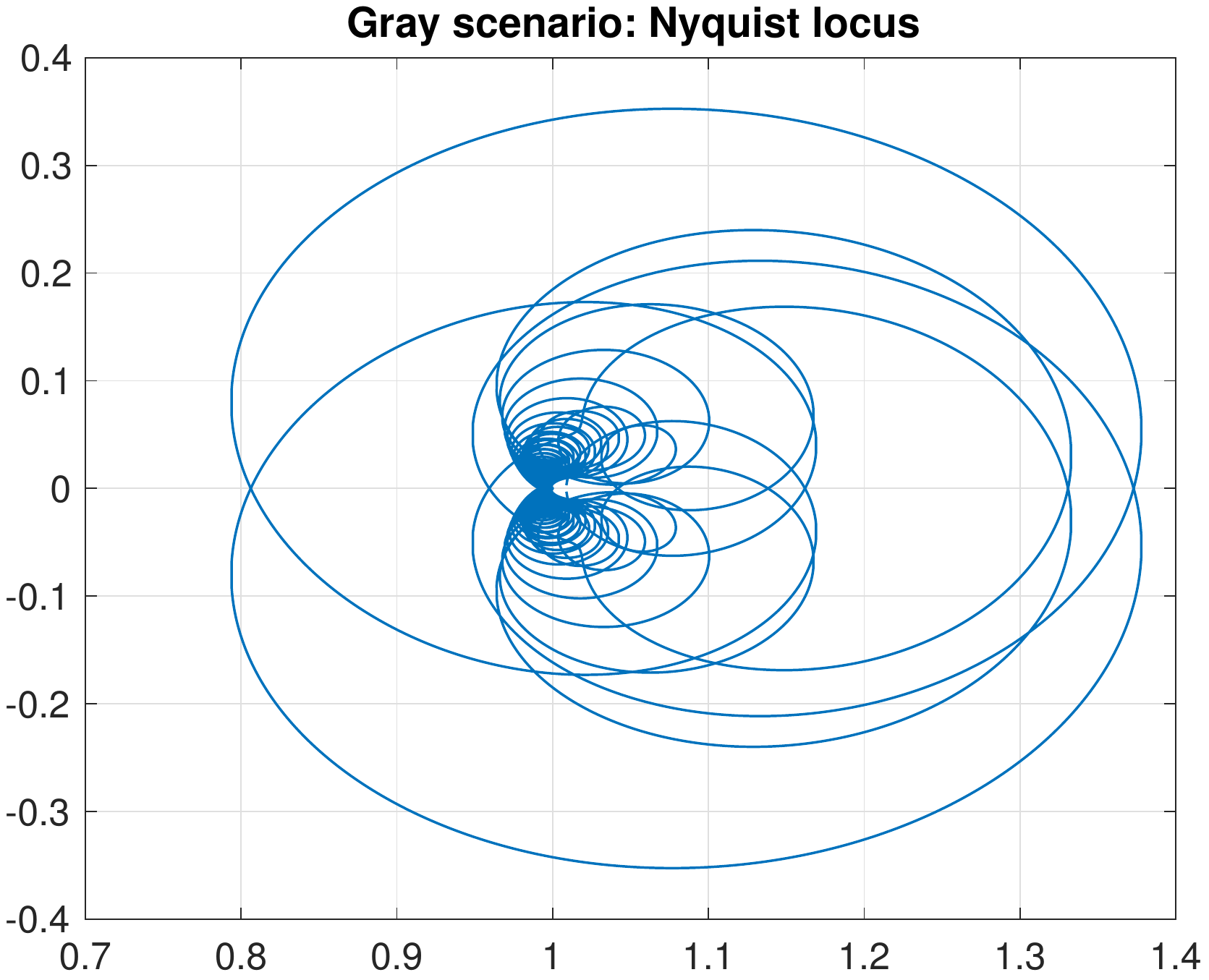}
\includegraphics[scale=0.4]{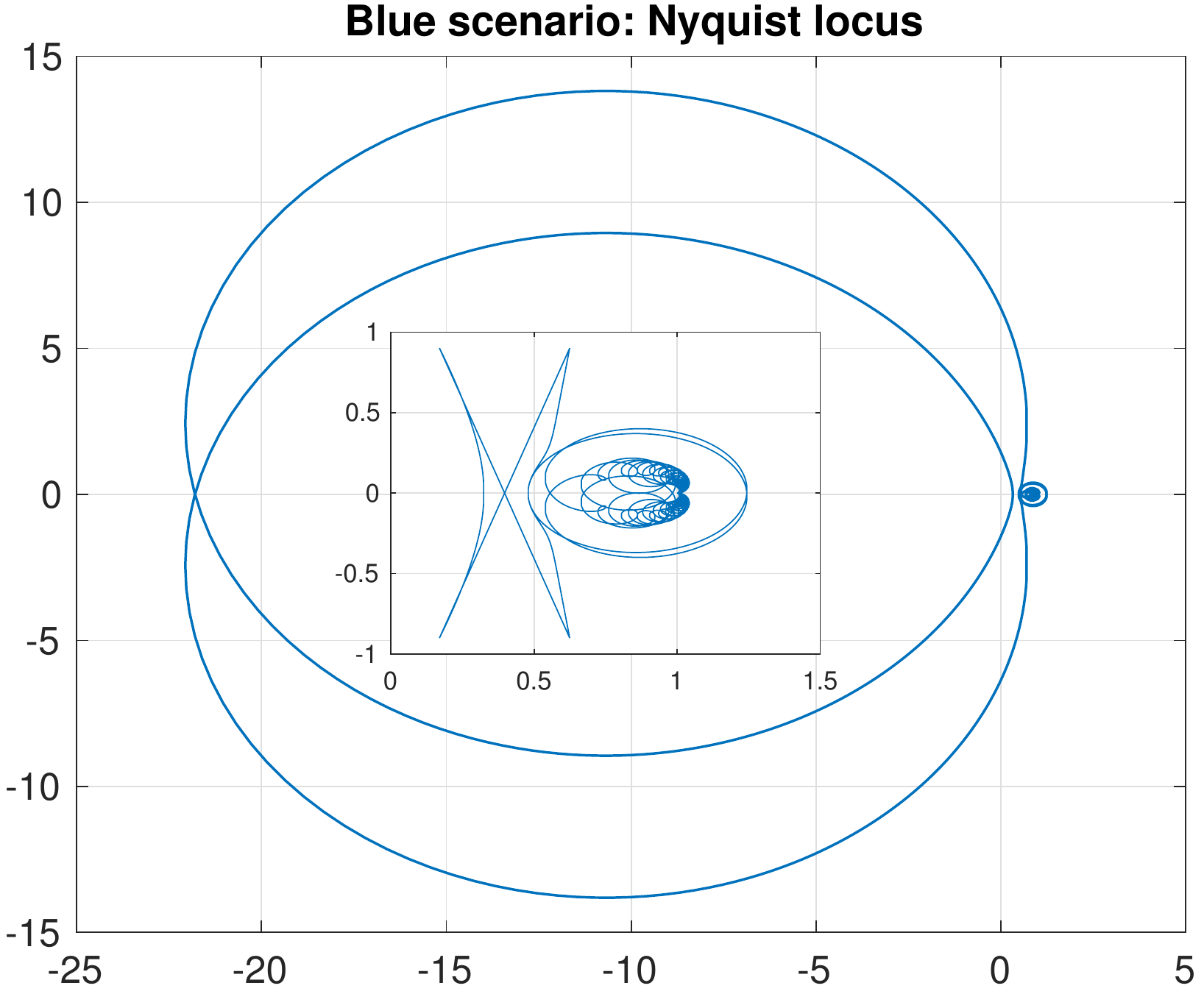}
\end{center}
\caption{Nyquist curve $1+K_{\rm blue}G_{\rm blue}$ (right) winds twice around origin. 
Since $K_{\rm blue}$ is stable and $n_p=2$, closed loop is certified exponentially stable. Gray case (left) has $n_p=0$
and winding number $0$ around origin (critical point).  Since $K_{\rm gray}$ is stable, loop is certified exponentially stable. \label{nyquist_blue}}
\end{figure}

As can be seen, in the blue case (right) the Nyquist curve winds twice around the origin. Since $K_{\rm blue}$ is stable and the open loop
$G_{\rm blue}$ has $n_p=2$ unstable poles, this proves exponential stability of the closed-loop    $(G_{\rm blue},K_{\rm blue})$.
At this point we have completed step 2 of the general synthesis algorithm \ref{algo1}. 

\begin{remark}
The fact that finite-difference and finite-element discretizations of stabilizable (or detectable) hyperbolic equations
may turn out not stabilizable (detectable) cannot be overcome by increasing $N$. 
This has been the cause of a large
body of controllability literature, which fortunately
has little relevance for control. Namely,
once we have decided that the true model
for the drilling process is the infinite-dimensional (\ref{system})-(\ref{outputs}), we little care whether  $K$,
synthesized for $G$ and $G_{\rm nl}$,
also stabilizes discretizations of $G$ or $G_{\rm nl}$. 
\end{remark}

\section{$H_\infty$-synthesis}
\label{synthesis}
The final step in algorithm \ref{algo1} is $H_\infty$-synthesis. While we have already shown that
the non-linear system can be locally exponentially stabilized by a finite-dimensional controller, we now strive to
prove global exponential stability of the closed-loop system 
$(G_{\rm nl},K)$.
In the following, it is helpful to represent the non-linear system $G_{\rm nl}$ in Lur'e form, i.e., as the closed loop interconnection of its linearization with a
static non-linearity.

\begin{figure}[!ht]
\centerline{
\includegraphics[scale=0.9]{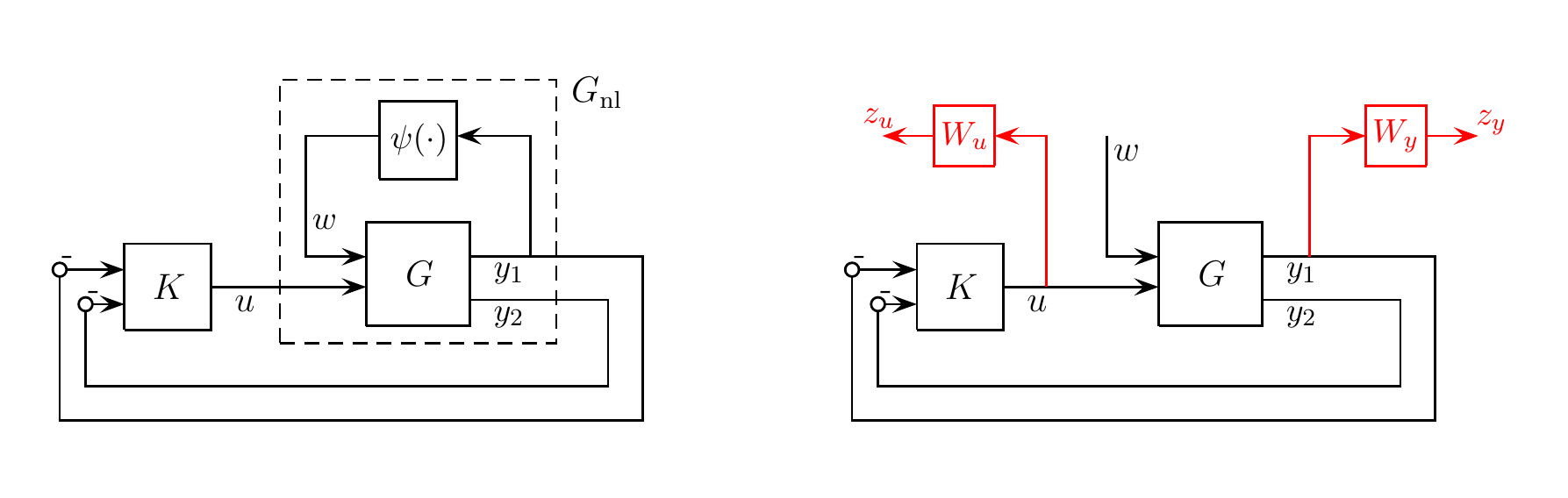}
}
\caption{Non-linear system (left) in feedback form. The synthesis interconnection (right)  interprets non-linearity as an exogenous  disturbance $w$. \label{fig_scheme}}
\end{figure}

\subsection{Mixed sensitivity}
The non-linear system can be consider as a feedback
loop between the linearized plant
\begin{align}
P:\qquad \qquad
\label{plant1}
\begin{split}
{x}_{tt}(\xi,t) &= x_{\xi\xi}(\xi,t) -2\lambda x_t(\xi,t) \\
x_\xi(1,t) &= -x_t(1,t) + u(t) \\
\alpha x_{tt}(0,t) &=x_\xi(0,t) +  qx_t(0,t) + w(t)\\
y_1(t) &=x_t(0,t), y_2(t) = x_t(1,t) , z = (y,u), 
\end{split}
\end{align}
connected with the controller
$u = Ky$ and the non-linearity $\psi(\cdot)$ as in Fig. \ref{fig_scheme} left. We now have several choices. The most straightforward one is to
grossly  interpret the non-linearity $\psi(x_t(0,t))$ as a disturbance $w$, forgetting its specific form.
In (\ref{plant1}) we then
introduce typical outputs
like $z_y=W_yy$, $z_u=W_uu$, where
the channel $w\to z_y$ rejects the effect of the non-linearity on the low-frequency part of the measured output, while
$w \to W_uu=z_u$ accounts for high frequency components of the control signal, so that
minimizing the $H_\infty$-norm of $T_{wz}(K)$ limits the degrading effects of the non-linearity while maintaining reasonable control authority. Here and for the following
$T_{ab}(K)$ denotes a closed-loop channel $b\to a$ in
plant $P$. 
The closed loop of (\ref{plant1}) with $K$ from $w$ to $z$ is obtained as 
$T_{zw}(K)= {\rm diag}(W_u,W_y) T_{(u,y),w}(K)$ as shown in  Fig. \ref{fig_scheme} right.

\subsection{Sector non-linearity}
A more sophisticated approach uses the fact that the non-linearity $\psi$ in
(\ref{NL}) induced by $\phi = \phi_{\rm mud} + \phi_{\rm rock}$ as in (\ref{rock}) is sectorial. That is to say, there exist
$q_l \leq q_u$ such that
$q_l \omega \leq \psi(\omega) \leq q_u\omega$ for all $\omega$,
i.e., (\ref{system}) is an infinite dimensional
Lur'e system. 
For the scenarios gray and blue these
sectors are shown in Fig. \ref{sector}.

\begin{figure}[htbp]
\begin{center}
\includegraphics[width=7cm, height=4.5cm]{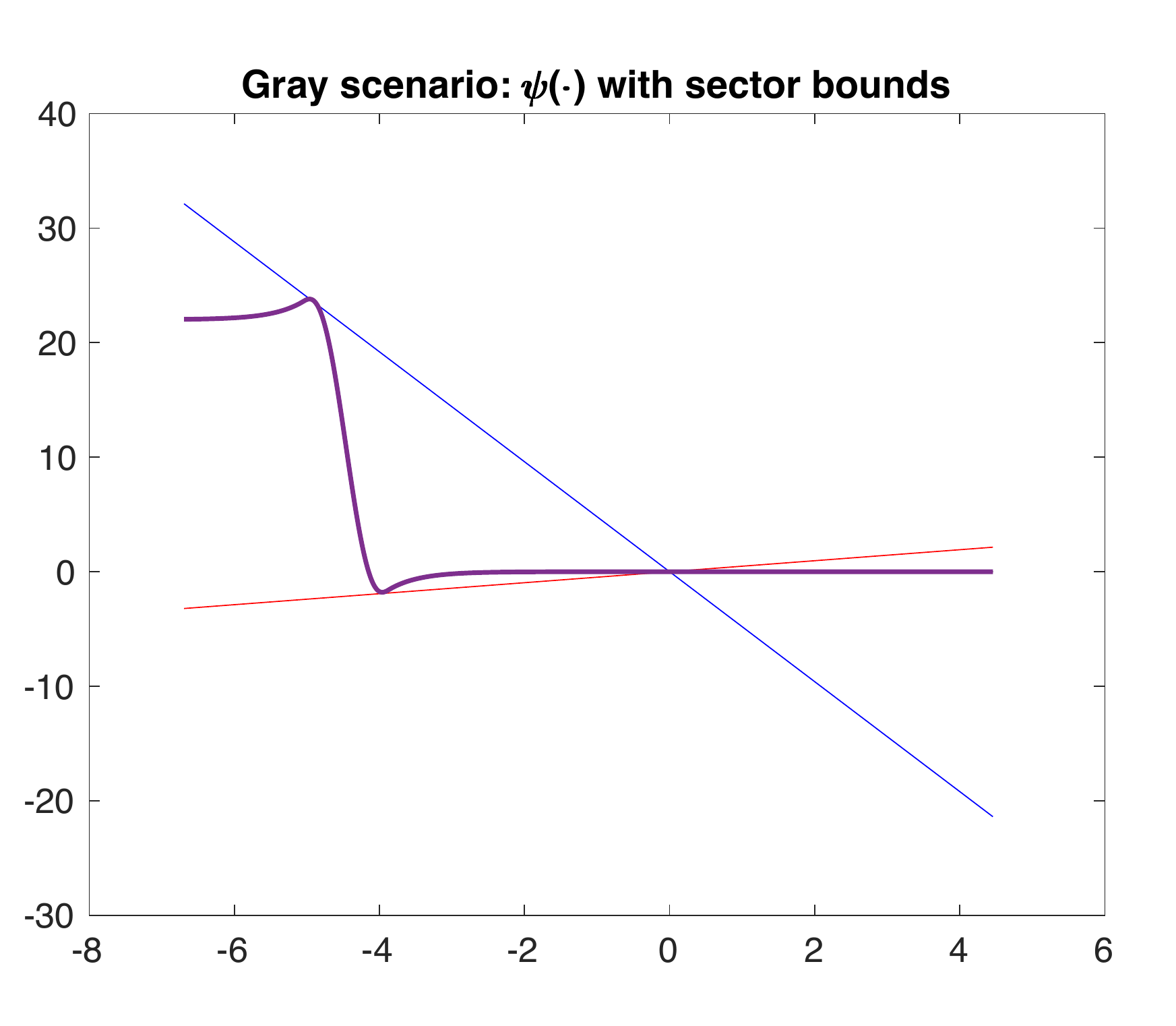}
\includegraphics[width=7cm,height=4.5cm]{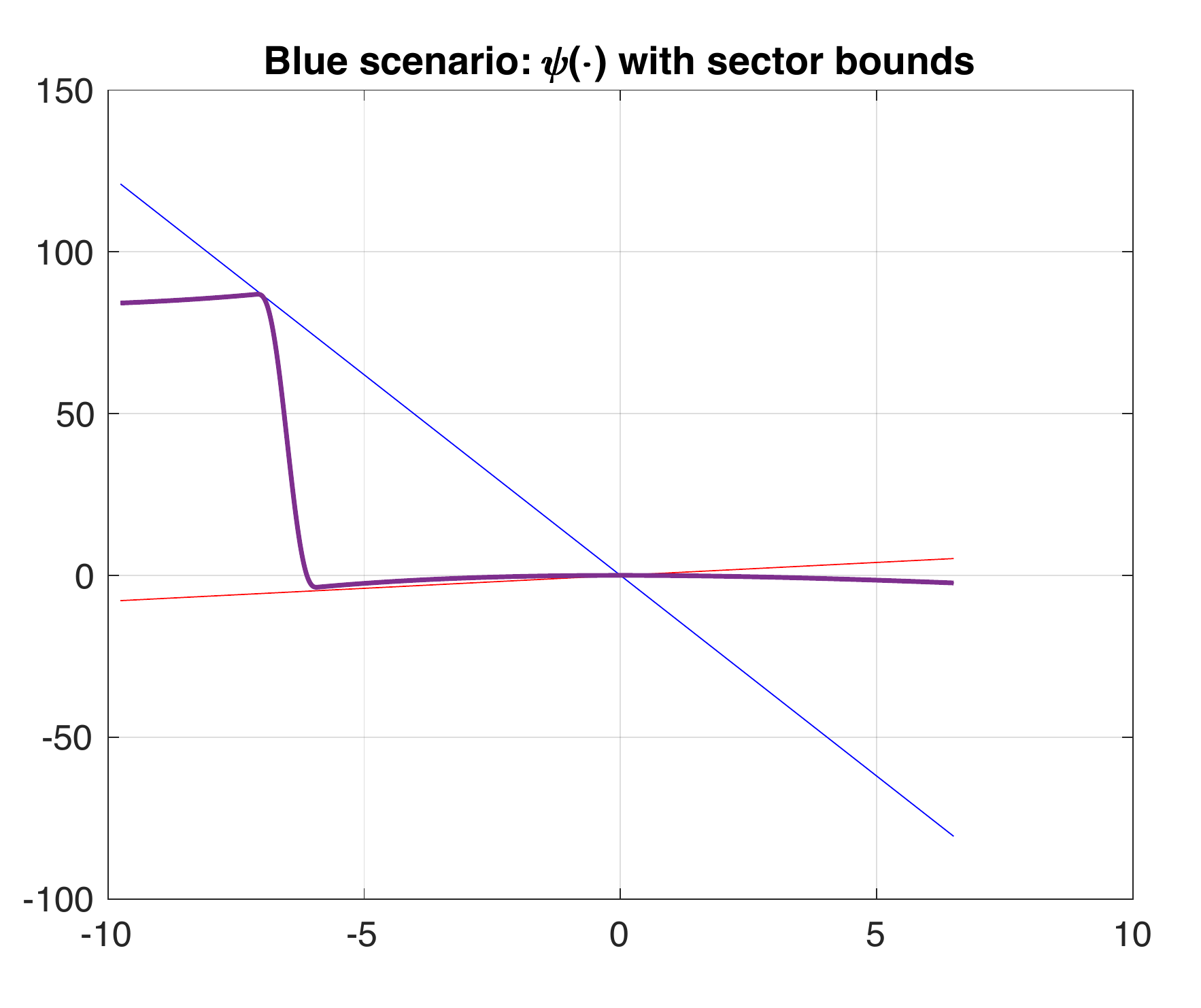}
\end{center}
\caption{Sector non-linearity $q_l \omega \leq \psi(\omega) \leq q_u\omega$ for gray scenario (left) and blue scenario (right). \label{sector}}
\end{figure}
\begin{lemma}
\label{lemma_sector}
For $\omega\to \pm\infty$ the non-linearity $\psi(\omega)$
behaves asymptotically like a line $-q_s \omega + a_\pm$,
where
\[
q_s = \frac{c_b}{\sqrt{GJI}} +q, \;
a_+= \frac{LW_{ob}R_b}{GJ}(\mu_{sb}-\mu_{cb})e^{-\frac{\gamma_b}{\nu_f}\Omega},\; a_-=2\frac{LW_{ob}R_b\mu_{cb}}{GJ}
+ \frac{L}{GJ} \left( \mu_{cb}-\mu_{sb}\right)e^{-\frac{\gamma_b}{\nu_f}\Omega}.
\]
\end{lemma}

\begin{proof}
Note that since we have transferred the steady state to the origin, the kink of the friction term $\psi$ occurs at $x_t=-\frac{L\Omega \sqrt{I}}{\sqrt{GJ}}=:-x_t^0$.
For $x_t \gg -x_t^0$ we have $\psi(x_t)=-(\frac{c_b}{\sqrt{GJI}}+q)x_t + 
\frac{LW_{ob}R_b}{GJ}(\mu_{sb}-\mu_{cb})e^{-\frac{\gamma_b}{\nu_f}\Omega} (1-e^{-\frac{\gamma_b}{\nu_f}\frac{1}{L}\sqrt{\frac{GJ}{I}}x_t}) 
\sim  -(\frac{c_b}{\sqrt{GJI}}+q)x_t + \frac{LW_{ob}R_b}{GJ}(\mu_{sb}-\mu_{cb})e^{-\frac{\gamma_b}{\nu_f}\Omega}= -q_sx_t+a_+
$,
and for $x_t < x_t^0$ we get
$\psi(x_t)=-(\frac{c_b}{\sqrt{GJI}}+q)x_t + 2 \frac{LW_{ob}R_b \mu_{cb}}{GJ}+ \frac{L}{GJ} \left(\mu_{cb}-\mu_{sb}\right)e^{-\frac{\gamma_b}{\nu_f}\Omega}   
+ \frac{L}{GJ} \left(\mu_{cb}-\mu_{sb} ) e^{-\frac{1}{L}\sqrt{\frac{GJ}{I}}|x_t|}e^{\frac{\gamma_b}{\nu_f} \Omega}  \right)\sim -(\frac{c_b}{\sqrt{GJI}}+q)x_t + 2 \frac{LW_{ob}R_b \mu_{cb}}{GJ}+ \frac{L}{GJ} \left(\mu_{cb}-\mu_{sb}\right)e^{-\frac{\gamma_b}{\nu_f}\Omega}
= -q_s+a_-$.
\hfill $\square$
\end{proof}

Since both branches
behave asymptotically like a line with slope 
\begin{equation}
    \label{slope}
    -q_s:=
-\frac{c_b}{\sqrt{GJI}}-q= -\frac{W_{ob}R_b (\gamma_b/\nu_f)(\mu_{sb}-\mu_{cb})e^{-\frac{\gamma_b}{\nu_f}\Omega}}{\sqrt{GJI}},
\end{equation}
it is not hard to find slopes
$q_l,q_u$ with $q_l\omega \leq \psi(\omega)\leq q_u\omega$.
Those can be seen in
Fig. \ref{sector} for the gray and blue cases. 
We use the standard notation $\psi\in {\bf sect}(q_l,q_u)$.

In order to achieve stability of the non-linear
closed loop, we now apply the technique
of Zames \cite{zames:66}, which requires that
the linear system $T_{y_1w}(K)$ 
in feedback with the non-linearity $\psi(\cdot)$
as in Fig. \ref{fig_scheme} right
satisfy the complementary sector constraint.
To put this to work, we let
$c=(q_l+q_u)/2$ and $r=(q_u-q_l)/2$, and introduce
the centered non-linearity
$\chi(w)=\psi(w)-cw$, which satisfies $\chi\in {\bf sect}(-r,r)$.

The centered non-linearity
$\chi(w) = \psi(w) - cw$ is now in
feedback with the following shifted plant:
\begin{align}
\widetilde{P}:\qquad\qquad
\label{plant}
\begin{split}
{x}_{tt}(\xi,t) &= x_{\xi\xi}(\xi,t) -2\lambda x_t(\xi,t) \\
x_\xi(1,t) &= -x_t(1,t) + u(t) \\
\alpha x_{tt}(0,t) &=x_\xi(0,t) +  (q+c) x_t(0,t) + e(t)\\
y_1(t) &=x_t(0,t), y_2(t) = x_t(1,t) , {z}(t) =x_t(0,t) , 
\end{split}
\end{align}
connected with 
\begin{equation}
\label{loop2}
u = Ky, \quad z_\chi = \chi(e_\chi), \quad e = z_\chi + w, \quad e_\chi = {z} + w_\chi.
\end{equation}
Closing the loop with regard to $u=Ky$ leads to ${z}=\widetilde{T}_{ze}(K)e$, which is in loop with the non-linearity
$z_\chi = \chi(e_\chi)$ as in Fig. \ref{zames_figure}. Here and in the following channels derived form plant $\widetilde{P}$
will be denoted $\widetilde{T}_{wz}(K)$ etc.
Note that the sole difference between $P$ and $\widetilde{P}$ is that the parameter $q$ is replaced by $\widetilde{q}=q+c$.
In particular, stabilization of $\widetilde{P}$ is obtained
as studied in section \ref{stabilize}. Ultimately this means that $K$ will have to stabilize the linear wave equation for two different values 
$q,\widetilde{q}$, while $\alpha,\lambda$ remain fixed.

\begin{figure}[!htbp]
\begin{center}
\includegraphics[scale=1.3]{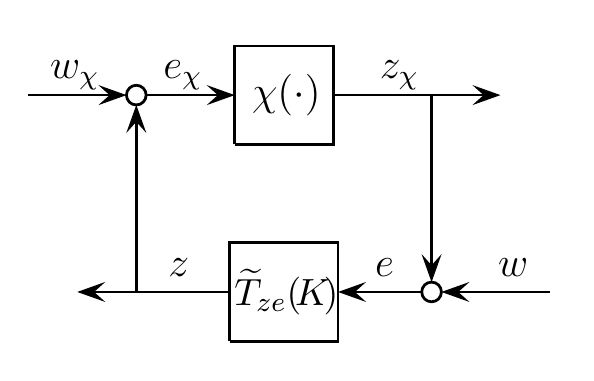}
\end{center}
\caption{Closing the loop with $u=Ky$ in (\ref{plant}) leaves an exponentially stable linear system
$\widetilde{T}_{ze}(K)$ in feedback with the shifted static non-linearity $z_\chi=\chi(e_\chi)$. 
\label{zames_figure}}
\end{figure}

\begin{lemma}
Let $\psi\in {\bf sect}(q_l,q_u)$ and put $c=(q_u+q_l)/2$, $r = (q_u-q_l)/2$. 
Suppose the controller $K$ has been tuned such that the closed loop $(\widetilde{P},K)$
is $H_\infty$-stable with 
$\|\widetilde{T}_{ze}(K)\|_\infty < r^{-1}$. Then the non-linear closed-loop {\rm (\ref{system})} with $u=Ky$ is finite gain input-output stable, i.e.,
there exists a constant $M>0$ such that in {\rm (\ref{plant})-(\ref{loop2})} we have
$\|x_t(0,\cdot)\|_2 + \|\psi(x_t(0,\cdot))\|_2 \leq M \left( \|w_\chi\|_2+ \|w\|_2\right)$ for all inputs $w,w_\chi\in L_2[0,\infty)$.
\end{lemma}

\begin{proof}
This follows from \cite[Thm. 1]{zames:66}. If $\psi \in {\bf sect}(q_l,q_u)$, then the centered non-linearity
$\chi:=\psi-cI$ satisfies $\chi\in {\bf sect}(-r,r)$, hence $\|\chi(z_\chi)\|_2 \leq r \|z_\chi\|_2$ has $L_2$-gain $r$ in the sense of
\cite[Def. (3)]{zames:66}. 
This non-linearity is now in feedback with $\widetilde{T}_{ze}(K)$.
Since by assumption $K$ has been tuned such that $\|\widetilde{T}_{ze}(K)\|_\infty < r^{-1}$, 
this LTI-system has $L_2$-gain $<r^{-1}$, and
the small gain
theorem implies boundedness of the loop (\ref{plant})-(\ref{loop2}), i.e., there exists a constant $M > 0$ 
such that $\|e_\chi\|_2 \leq M (\|w\|_2+ \|w_\chi\|_2)$ and $\|e\|_2 \leq M(\|w\|_2+\|w_\chi\|_2)$ in Fig. \ref{zames_figure}. 
Since in closed loop the input $e$ to $\widetilde{T}_{ze}(K)$ represents the non-linear term $\chi(x_t)$, we derive
$\|\chi(x_t(0,\cdot))\|_2 \leq M\left(\|w_\chi\|_2+\|w\|_2\right)$.
Still from the small gain theorem we get $\|e_\chi\|_2 \leq M\left( \|w_\chi\|_2+\|w\|_2\right)$, and since in closed loop $e_\chi$ represents the output $x_t(0,\cdot)$, we have
$\|x_t(0,\cdot)\|_2 \leq M\left(\|w_\chi\|_2+\|w\|_2\right)$ in closed loop. Finally, for $\psi = \chi + cI$ we get a similar estimate by combining the previous two:
$\|\psi(x_t(0,\cdot))\|_2 \leq \|\chi(x_t(0,\cdot))\|_2 + c\|x_t(0,\cdot)\|_2 \leq M(1+c)\left(\|w_\chi\|_2+\|w\|_2\right)$.
\hfill $\square$
\end{proof}

This has now the following consequence:

\begin{proposition}
Let $\psi\in {\bf sect}(q_l,q_u)$ with $c,r$ as above, and
suppose the controller $K$ has been tuned such that the closed loop 
$(\widetilde{P},K)$ is $H_\infty$-stable with
$\|\widetilde{T}_{ze}(K)\|_\infty < r^{-1}$. Then the non-linear closed loop between {\rm (\ref{system})}
and $u = Ky$ is input-to-state stable in the following sense: If the input signal $w\in L_2[0,\infty)$, then the
state $(x,x_t)$ of the non-linear closed loop with initial condition $x_{cl}(0)=x_0$ is in $L_2([0,\infty),H)$. 
\end{proposition}

\begin{proof}
Write the non-linear closed loop in the abstract state-space $H$ in theorem \ref{theorem2}
as $\dot{x}_{cl} = A_{cl} x_{cl} + \Psi(x_{cl}) + w$, $x_{cl}(0)=x_0$, where $A_{cl}$ is exponentially stable,
$\Psi(x_{cl}) = \psi(x_t(0,\cdot))$ for the closed-loop $x_t(0,t)$, and where $w(t)$ is an input to the equation 
$\alpha x_{tt}(0,t) =x_\xi(0,t) + qx_t(0,t)+\psi(x_t(0,t))$. 
The linear feedback system $(\widetilde{P},K)$, respectively its channel $\widetilde{T}_{ze}(K)$, 
is now $\dot{x}_{cl}=A_{cl}x_{cl} + e$, ${z}=C_{cl}x_{cl}$,
in loop with the centered non-linearity $\chi(\cdot)$, and $w$ is the lower right input in Fig. \ref{zames_figure}.
To account for a non-zero initial condition
$x_{cl}(0)=x_0$ we choose the top left input in Fig. \ref{zames_figure} as
$w_\chi = C_{cl} e^{A_{cl}t}x_0$, where $C_{cl}$ is the output operator of closed loop system $(\widetilde{P},K)$. Then $e_\chi$ is the solution of the
Cauchy problem $\dot{x}_{cl}=A_{cl}x_{cl}+e, x_{cl}(0)=x_0$.
From the lemma we get $\|\Psi(x_{cl})\|_2 = \|\psi(x_t)\|_2 \leq M\left( \|w_\chi\|_2+\|w\|_2\right) \leq M' \left(\|x_{cl}(0)e^{-\omega_0t}\|_2 +\|w\|_2  \right)$,
where $-\omega_0 < 0$ is the growth rate of the exponentially stable generator $A_{cl}$.
In particular, if we put $v(t) = \Psi(x_{cl}(t))+w(t)$, then $\|v\|_2 \leq (M'+1) \left(|x_{cl}(0)|+  \|w\|_2\right)$, hence we may consider
$v(t)$ as a right hand side in $L_2$ to the
non-homogeneous Cauchy problem $\dot{x}_{cl}=A_{cl}x_{cl} + v$, $x_{cl}(0)=x_0$. Since $A_{cl}$ is exponentially stable, the closed loop state
is then also in $L_2$; \cite[Ch. VI,7.1a]{engel_nagel}.
\hfill $\square$
\end{proof}

One wonders whether the state $x_{cl}(t)$ decays exponentially to 0 when this is the case for the input $w(t)$. 
Suppose $w \in L_2$ decays exponentially in the sense that $w=e^{-at} \widetilde{w}$ for some $a > 0$ and $\widetilde{w}\in L_2$. 
In this case it seems plausible to work with the weighted $L_2$-norm
$\interleave w\interleave_2^2= \|e^{at} w(\cdot)\|_2^2=\|\widetilde{w}\|_2^2$.

\begin{proposition}
\label{exponentially}
Suppose $\psi\in {\bf sect}(q_l,q_u)$ with $c,r$ as above, and suppose $K$ has been tuned such that
$(\widetilde{P},K)$ is $H_\infty$-stable, with $\|\widetilde{T}_{ze}(K)\|_\infty < r^{-1}$. There exists $a > 0$ such that
whenever the input
$w$ decays exponentially with rate at least as fast as $a$, i.e., $w(t) = e^{-at}\widetilde{w}(t)$ for some $\widetilde{w}\in L_2[0,\infty)$,
then the state $x_{cl}(t)$ of the non-linear closed loop in response to the input $w$ decays exponentially with rate
at least  $a$.
\end{proposition}

\begin{proof}
1)
Since the closed loop $(\widetilde{P},K)$ is exponentially stable with 
$-\omega_0 := \omega_0(A_{cl}) < 0$ and
$\|\widetilde{T}_{ze}(K)\|_\infty < r^{-1}$, we may
choose a small enough shift $0 < a < \omega_0$
such that $(\widetilde{P}(\cdot-a),K(\cdot-a))$ is still exponentially stable and
$\|\widetilde{T}_{ze}(K)(\cdot -a)\|_\infty < r^{-1}$. 
Let $\interleave \cdot \interleave_2$ be the corresponding weighted $L_2$-norm as above.

2)
Let us observe that for the centered non-linearity
$\chi\in {\bf sect}(-r,r)$ implies $\interleave \chi(w)\interleave_2 \leq r \interleave w\interleave_2$
for all $w = e^{at} \widetilde{w}$. Namely,
$\interleave\chi(w)\interleave_2^2=\int_0^t  e^{2a\tau} |\chi(w(\tau))|^2 d\tau \leq \int_0^t e^{2a\tau} r^2 |w(\tau)|^2d\tau
= r^2 \interleave w\interleave_2^2$.

3)
Now we establish the complementary estimate for the LTI 
feedback system $(\widetilde{P},K)$ and its channel $\widetilde{T}_{ze}(K)$  
with regard to the norm
$\interleave \cdot \interleave_2$. 
We have
\begin{align*}
    \interleave \widetilde{T}_{ze}(K) \ast w\interleave _2^2 &= \int_0^\infty e^{2at} \left| \int_0^t \widetilde{T}_{ze}(K)(t-\tau) w(\tau) d\tau\right|^2 dt\\
    &= \int_0^\infty  \left| \int_0^t \widetilde{T}_{ze}(K)(t-\tau) e^{a(t-\tau)} w(\tau)e^{a\tau} d\tau\right|^2 dt\\
    &= \int_0^\infty \int_0^t \left| \left(\widetilde{T}_{ze}(K)\cdot e^{at}\right)(t-\tau) \left( {w}\cdot e^{at}\right)(\tau)d\tau \right|^2 dt\\
    &= \| (\widetilde{T}_{ze}(K)\cdot e^{at}) \ast \widetilde{w} \|_2^2 = \| \widetilde{T}_{ze}(K)(s-a)\cdot \widetilde{w}(s)\|_2^2\\
    &\leq \|\widetilde{T}_{ze}(K)(\cdot-a)\|_\infty^2 \|\widetilde{w}\|_2^2 = \|\widetilde{T}_{ze}(K)(\cdot-a)\|_\infty^2 \interleave w\interleave_2^2\\
    &< r^{-2} \interleave w\interleave_2^2.
\end{align*}
This means we may apply the small gain argument with the norm $\interleave\cdot\interleave_2$. The result is as before
that $\interleave x_t(0,\cdot)\interleave_2 + \interleave \psi(x_t(0,\cdot))\interleave_2 \leq M
\left(\interleave w_\chi\interleave_2+\interleave w \interleave_2\right)$
for some $M>0$ and all inputs $w=e^{-at}\widetilde{w}$,
$w_\chi = e^{-at} \widetilde{w}_\chi$
with $\widetilde{w},\widetilde{w}_\chi\in L_2[0,\infty)$. That means the non-linearity in closed loop
in response to the signal $w=e^{-at}\widetilde{w}$ also decays at least as fast as $e^{-at}$, so that the right hand side
$v(t) = \Psi(x_{cl}(t)) + w(t)$ already used in the previous proposition is of the form $v(t)=e^{-at}\widetilde{v}(t)$ for some $\widetilde{v}\in L_2$. 

We also have to argue that $w_\chi =C_{cl}e^{A_{cl}t}x_0$ decays with rate $a$, which
holds since $a < -\omega_0(A)$.
But now all we have to observe is that due to exponential stability of $A_{cl}$  in the non-homogeneous Cauchy problem
$\dot{x}_{cl} = A_{cl} x_{cl} +v$ the state decays exponentially as soon as $v$ decays exponentially. The mild solution
in the semi-group sense 
\cite[p. 436]{engel_nagel} satisfies $x_{cl}(t)= e^{A_{cl}t}x_{cl}(0) + \int_0^t e^{A_{cl}(t-\tau)} v(\tau) d\tau$, 
hence $|x_{cl}(t)|\leq M\left(e^{-\omega_0t}+ \|\widetilde{v}\|_2 \int_0^t e^{-\omega_0 (t-\tau)}e^{-a\tau} d\tau \right)\leq M(1+\|\widetilde{v}\|_2/(\omega_0-a)) e^{-at}$.
\hfill $\square$
\end{proof}

This brings us now to our first
optimization program, where we combine a mixed $H_\infty$
performance and robustness
requirement (Fig. \ref{fig_scheme} right)  for the nominal plant with a sector constraint assuring global exponential stability
of the non-linear closed loop (Fig. \ref{fig_scheme} left) when satisfied:

\begin{eqnarray}
\label{program1}&
\begin{array}{ll}
\mbox{minimize} & r\|\widetilde{T}_{{z}e}(K)\|_\infty \\  
\mbox{subject to} & \|W_u T_{uw}(K)\|_\infty \leq 1 \\
&K \in \mathscr K
\end{array}
\end{eqnarray}
Here $\mathscr K$ refers to a class of structured
controllers, and optimization over $K\in \mathscr K$ can be dispensed with as soon as the objective attains a value $<1$. 
As our experiments show,
the sectorial approach works successfully for the gray scenario.
Note that it is implicit in (\ref{program1})
that $K$ has to stabilize $P$ and $\widetilde{P}$,
which means stabilizing the wave equation
for the two different values $q$ and $\widetilde{q}=q+c$ with the same $\alpha,\lambda$.

\subsection{Large magnitude sector constraint}
The
limitation of the sector approach is 
obviously 
that if the primal sector ${\bf sect}(q_l,q_u)$ is large, 
it is difficult to tune $K$ such that the closed
loop system $({P},K)$ is in the complementary sector.  In the transformed metric, if the primal sector is large, then 
$r$ is large, so $r^{-1}$ is small and the constraint $\|\widetilde{T}_{ze}(K)\|_\infty < r^{-1}$ 
in (\ref{program1})
is difficult to achieve -- if at all.
This fails indeed for the blue scenario,
and Zames-Falb multipliers \cite{zames_falb} do not help for the specific non-linearity 
$\psi$.
However, the particular structure of the non-linearity
in Lemma \ref{lemma_sector}
suggests the following definition
as a remedy.

We say that $\psi$ satisfies a {\it large magnitude
sector}  constraint, denoted
$\psi \sim {\bf sect}(q_l,q_u)$, if there exist constants $L,M > 0$ such that
$\left(\psi(x)-q_lx\right)\cdot\left( \psi(x) - q_ux\right) \geq 0$ for all $|x| > M$, while
$|\psi(x)| \leq L|x|$ for $|x|\leq M$. A {\it strict}  large magnitude sector is defined analogously. This is indeed what happens
for $\psi(\cdot)$ here, because from Lemma \ref{lemma_sector}
it follows that
any choice $q_l < -q_s < q_u$ will give such a
large magnitude sector.

\begin{proposition}
\label{asymp}
Suppose $\psi$ satisfies a large magnitude sector constraint
$\psi\sim {\bf sect}(q_l,q_u)$ with constants $M,L$. Let $c=(q_u+q_l)/2$, $r = (q_u-q_r)/2$, and suppose
the controller $K$ has been tuned such that the loop
$(\widetilde{P},K)$ is $H_\infty$-stable and satisfies
$\|\widetilde{T}_{ze}(K)\|_{{\rm pk}\_{\rm gn}} < r^{-1}$ for the peak-to-peak norm. Then for every input $w\in L_\infty[0,\infty)$ the non-linear 
closed loop state trajectory
$x_{cl}(t)$ is in $L_\infty([0,\infty),H)$.
\end{proposition}

\begin{proof}
1)
As before let $\chi = \psi - cI$ be centered, 
then $|\chi(x)| \leq r|x|$ for all $|x| > M$, while 
$|\chi(x)|\leq (L+c)|x|$ for $|x|\leq M$. 
We show that this implies
$|\chi(w)|_\infty \leq r|w|_\infty + k$ for some constant $k>0$ and all $w\in L_\infty[0,\infty)$ in the time domain.
Indeed,
\begin{align*}
\sup_{t > 0} |\chi(w(t))| &\leq \sup_{|w(t)|>M} |\chi(w(t))| + \sup_{|w(t)|\leq M} |\chi(w(t))|\\
&\leq \sup_{|w(t)|>M} r|w(t)| + \sup_{|w(t)|\leq M} (L+c)|w(t)|\\
&\leq r |w|_\infty + (L+c)M =: r |w|_\infty + k.
\end{align*}
Note that the same also holds in the truncated version, i.e.,
$|\chi(w)\cdot{\bf 1}_{[0,t]}|_\infty \leq r|w\cdot {\bf 1}_{[0,t]}|_\infty + k$ for every $t > 0$ and all $w$.

2)
Note that $\|\widetilde{T}_{ze}(K)\|_{{\rm pk}\_{\rm gn}} < r^{-1}$  means $| \widetilde{T}_{ze}(K)\ast w|_\infty \leq (r^{-1}-\delta)|w|_\infty$ for some small $\delta > 0$
with respect to the time-domain
space $L_\infty[0,\infty)$, and similarly in the truncated version.

3) But now both $\widetilde{T}_{ze}(K)$ and the non-linearity $\chi(\cdot)$ are finite-gain stable in the 
sense e.g. of \cite[Def. 3]{mareels_hill:92} with regard to $|\cdot|_\infty$. Namely
$|\chi(w)|_\infty \leq r|w|_\infty + k$ and $|\widetilde{T}_{ze}(K)\ast w|_\infty \leq (r^{-1}-\delta)|w|_\infty$, both fully and in the truncated version.
Since $r\cdot (r^{-1} -\delta)< 1$, it follows from \cite[Cor. 1]{mareels_hill:92} that the closed loop
of Fig. \ref{zames_figure} is finite-gain stable in the sense that
$|z|_\infty \leq M (|w|_\infty + |w_\psi|_\infty) + k$ and $|z_\psi|_\infty \leq M (|w|_\infty + |w_\psi|_\infty) + k$ 
for certain $M,k>0$.  We derive as before
that $|x_t(0,\cdot)|_\infty \leq M(|w|_\infty + |w_\chi|_\infty) + k$
and $|\psi(x_t(0,\cdot))|_\infty \leq M(|w|_\infty+|w_\chi|_\infty) + k$ for all $w\in L_\infty$, where $x_t(0,\cdot)$ is with regard to the closed loop. 

4) Putting $\Psi(x_{cl}(t)) = \psi(x_t(0,t))$ and $v(t) = \Psi(x_{cl}(t))+w(t)$ as before, we can consider $v(t)$ as a right hand side in
the non-homogeneous Cauchy problem $\dot{x}_{cl} = A_{cl}x_{cl} + v$. 
Accounting for non-zero initial data needs
$w_\chi(t) = C_{cl}e^{A_{cl}t}x_0$.
Since $w\in L_\infty$,
we have $|v\cdot {\bf 1}_{[0,t]}|_\infty \leq (M+1)(|w|_\infty +
|w_\chi|_\infty)+k=:k'$ for all $t$, and since $v$ is square integrable up to time $t$,
i.e., $v\cdot{\bf 1}_{[0,t]}\in L_2[0,t]$, 
the solution $x_{cl}$ exists on $[0,t]$ and is bounded independently of $t$ by a constant depending only on $k'$
and the decay rate $\omega_0(A_{cl})$ of $A_{cl}$.
This gives $x_{cl} \in L_\infty$ as desired, and the solution exists at all times $t>0$.
\hfill $\square$
\end{proof}

\begin{remark}
It is clear that the impact of this result hinges on
computing $K$ for a sufficiently large  sector where
the constant $k$ is as small as possible, as that controls
how far the trajectory $x_{cl}(t)$ may remove herself from the steady state $0$. 
\end{remark}

\subsection{Overshoot}
It has been suggested  in the literature
that slip-stick is avoided as soon as the non-linear
system is globally stabilized. This is obviously misleading,
as any sufficiently strong disturbance will cause
the trajectory $x_t$ to attain the value
$-x_t^0$, however stable the loop. Stability would then only 
make the difference that the
trajectory, after being stuck,  returns to steady state when the effect of the disturbance ceases, while an unstable design might remain
stuck.
Since the non-linearity $\psi(\cdot)$ is concave in the neighborhood
of $0$, 
the term $qx_t+\psi(x_t)= (q+\frac{1}{2}px_t)x_t+{\rm o}(x_t^2)< qx_t$ is
slightly below the linearized term $qx_t$, so that a
linear controller may overestimate
its effect. This may cause overshoot in the 
response to a disturbance, thereby
increasing the risk of slip-stick.
That suggests
optimizing the closed loop against overshoot
in the channel $w \to y_1$, which we realize by simply minimizing the (unweighted) $H_\infty$-norm of $T_{y_1w}(K)$. Reduction of  peak-gain over frequency is known to be a suitable approach for systems with dominant second-order characteristics and  performs equally well in the present case. 
\noindent
In combination with the large magnitude sector 
this leads now to the program
\begin{eqnarray}\label{pgL1}
\begin{array}{ll}
\mbox{minimize} & \| T_{y_1w}(K) \|_\infty\\
\mbox{subject to} & \|\widetilde{T}_{{z}e}(K)\|_{{\rm pk}\_{\rm gn}} \leq 1/r\\
&\|W_uT_{uw}(K)\|_\infty \leq 1\\
&K \in \mathscr K
\end{array}
\end{eqnarray}
where $T_{y_1w}(K)$ is the closed loop transfer
$w \to y_1$ obtained from plant $P$,  
$\widetilde{T}_{{z}e}(K)$ refers to the transfer $e\to {z}$
in plant $\widetilde{P}$, and the channel $w\to z_u$ in plant $P$ is a safeguard against unrealistic control actions. 
This leads to satisfactory
results in the blue case, even though the stability 
certificate is weaker in the sense that
the non-linear closed loop trajectory $x_{cl}(t)$ is only
guaranteed locally exponentially stable and globally bounded.  

\begin{remark}
The peak-gain or peak-to-peak norm $\|\cdot\|_{{\rm pk}\_{\rm gn}}$
is the time domain $L_\infty$-operator norm, which for SISO systems
is equal to the time-domain $L_1$-norm of the impulse response, or the total variation of the 
step response \cite[Sect. 5.2]{boyd_barratt}. It is harder to compute, let alone to optimize, than the 
$H_\infty$-norm, but the bound $\|H\|_\infty \leq \|H\|_{{\rm pk}\_{\rm gn}}$ is known.  
Non-smooth analysis
of $\|\cdot\|_{{\rm pk}\_{\rm gn}}$ is beyond the scope of this
work and will be presented elsewhere. In our experiments
we use the trapezoidal rule to estimate the integral
of the absolute value of the impulse response
of $(\widetilde{P},K)$, and a heuristic to optimize it. Bounds for $\|\cdot\|_{{\rm pk}\_{\rm gn}}$ have been discussed e.g. in \cite{BB:92}, and
a minimization approach via linear programming is discussed
in \cite{diaz} for the case of full order (unstructured)
$K$.
\end{remark}

\section{Experiments}
\label{numerics}
\subsection{Gray scenario}
The gray scenario has been addressed with the approach
(\ref{program1}), where $q_l=-4.8$, $q_u=-4.8$, 
$W_u=\frac{1\mathrm{e}4 s }{s + 2\mathrm{e}5}$. 
Using Kalman reduction  to determine a minimal realization, the finite-difference model with $N=50$
is used to design a preliminary controller
$K_0 \in \mathscr K_5$ in the class of $5^{th}$-order
controllers. The Nyquist test
\cite[Thm. 1]{AN:18} shows that $K_0$ already
stabilizes the linear infinite
dimensional loop exponentially. Moreover, $K_0$
satisfies the sector constraint $\|\widetilde{T}_{{z}e}(K_0)\|_\infty= 0.281 < r^{-1}=1/2.64=0.379$ strictly. After choosing a small enough
tolerance with $\|\widetilde{T}_{{z}e}(K)\|_\infty + \vartheta < r^{-1}$, we check using \cite[Thm. 2]{AN:18} that
$K_0$ satisfies even the 
infinite dimensional sector constraint, so that
the non-linear closed loop
$(G_{\rm nl},K_0)$ is {\em proved} globally exponentially
stable in the sense of Proposition \ref{exponentially}.

In a second phase this controller is further
optimized with the true infinite dimensional
system as described in \cite{AN:18}, maintaining the stability and performance
certificates already achieved during optimization. Ultimately
this leads to the controller
$K_{\rm gray}\in \mathscr K_5$ in (\ref{controllers}) which has the same stability
certificates, and  slightly improved $H_\infty$-performance. This controller was
then tested in non-linear simulations with
spatial discretizations $N=200$. For instance,
in Fig. \ref{test1} (left) an initial condition
$\theta_t(0)< \theta_t^0=\Omega$ representing a deviation of $60\%$ from the steady-state was chosen. The
controller was switched on at time $t=10$ and simulated with a square-wave  disturbance 
occurring at $t=15$ with magnitude $60\%$ of the steady-state. In the gray scenario linear and non-linear
trajectories are almost identical. That slip-stick
may still occur even for this highly stable scenario
is seen in Fig. \ref{noisy} (right), but due to stability
the trajectory $\theta_t$ is able to free herself
and regain speed.

\begin{figure}[h]
\includegraphics[scale=0.45]{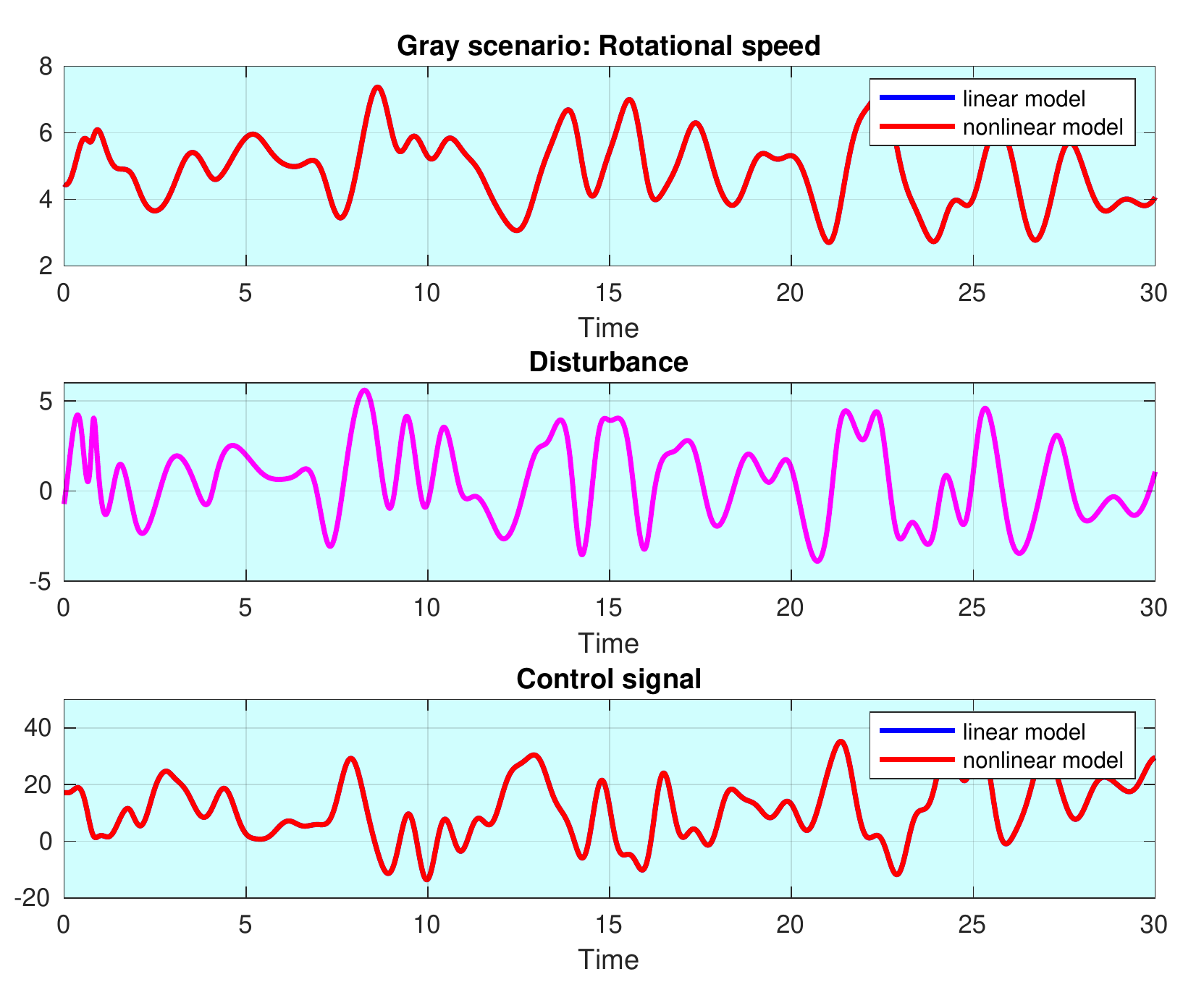}
\includegraphics[scale=0.45]{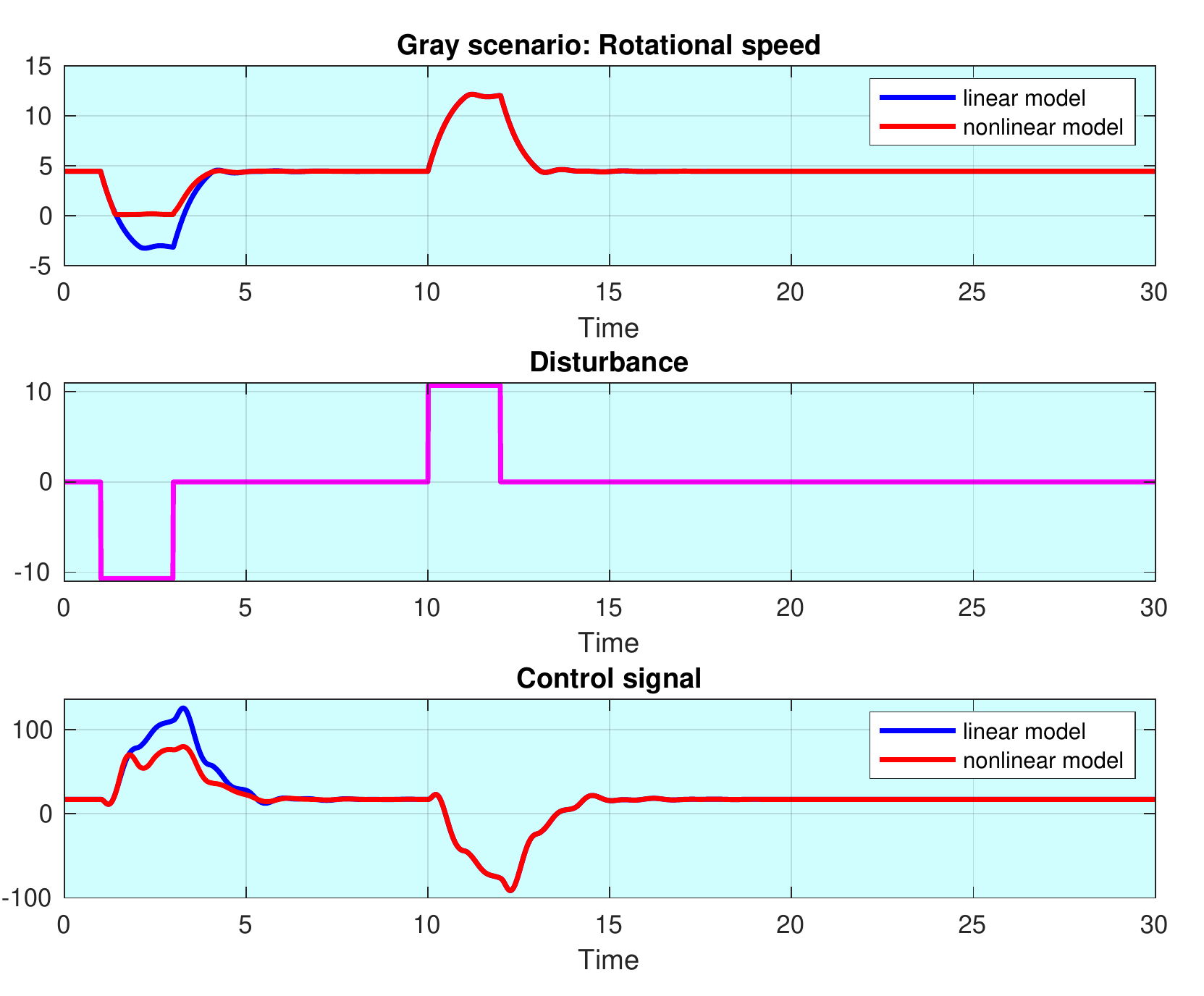}
\caption{Gray scenario: Occasional slip-stick occurs even with global stability.  Oscillatory
disturbance (left). Disturbance at $t=3,10$ (right).\label{noisy}}
\end{figure}

\begin{figure}[h]
\includegraphics[scale=0.45]{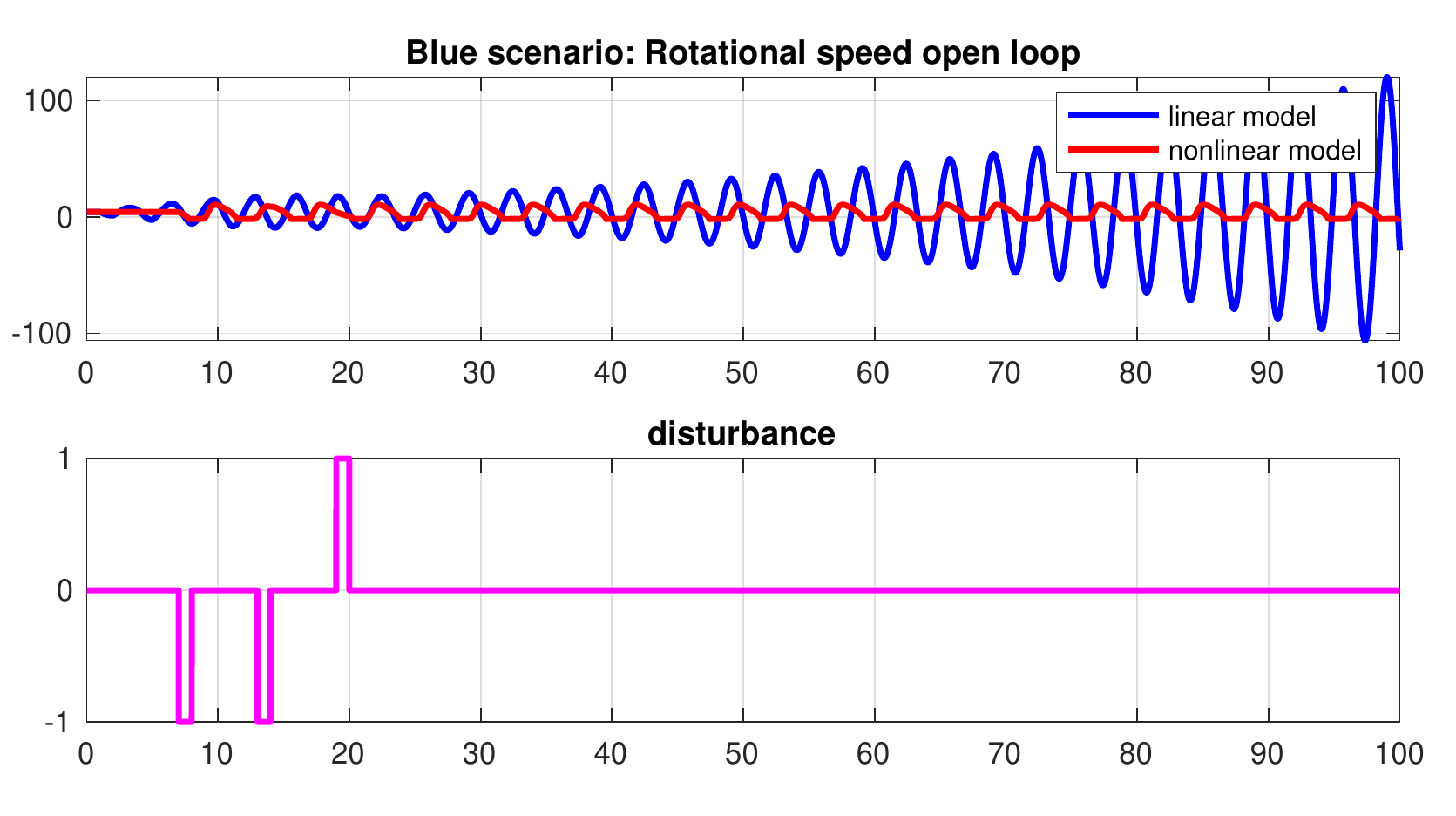}
\includegraphics[scale=0.45]{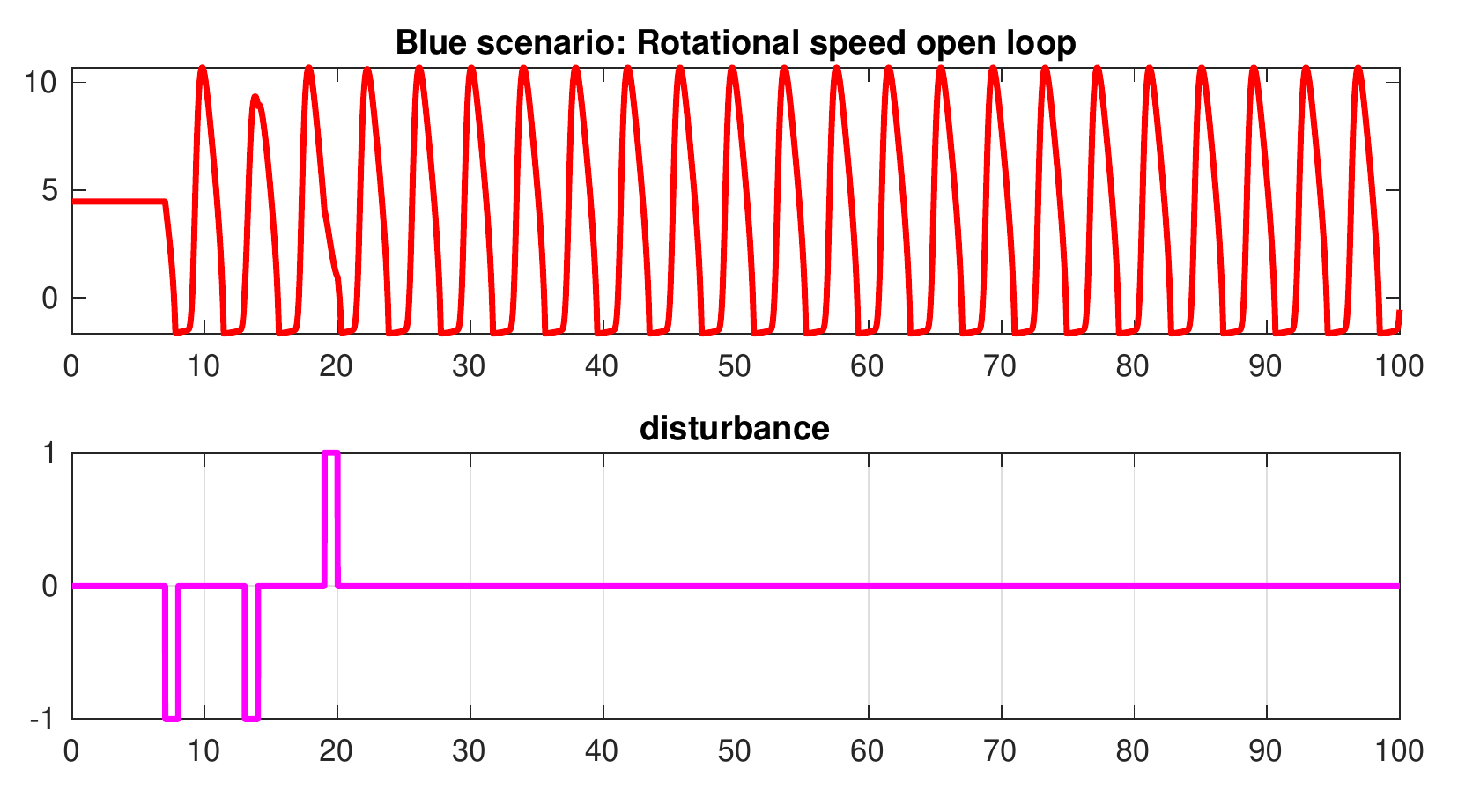}
\caption{Blue scenario: Slip-stick in open loop. \label{openloop}}
\end{figure}

\begin{figure}[h]
\includegraphics[width=7.5cm,height=7cm]{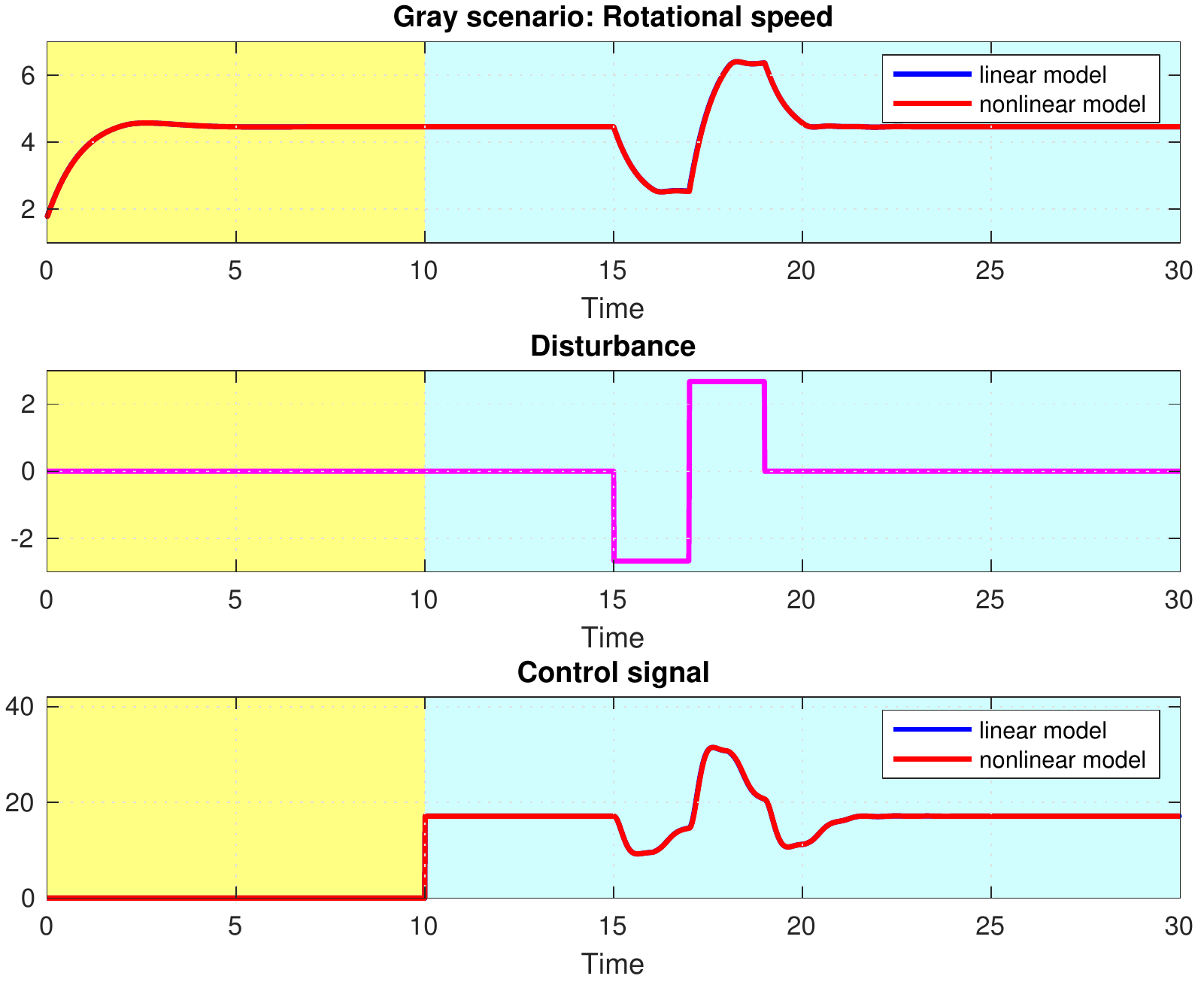}
\includegraphics[width=7.5cm,height=7cm]{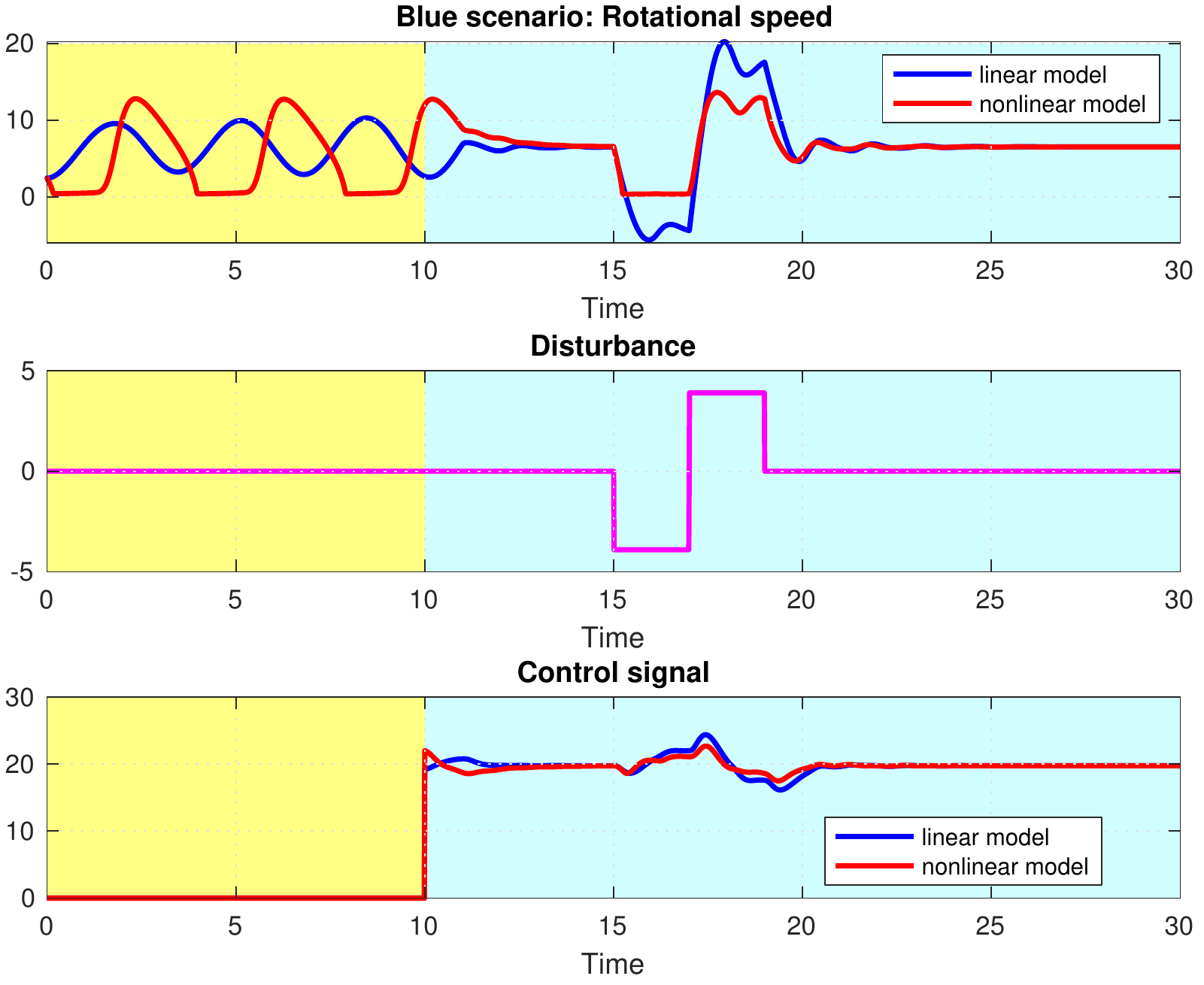}
\caption{Initial value below steady state, control switched on at $t = 10$. Disturbance at $t=15$. Gray left, blue right. \label{test1}}
\end{figure}

\begin{align}
\label{controllers}
\begin{split}
\tiny K_{\rm gray} &= \left[\begin{array}{l|l}A_K & B_K \\\hline C_K & D_K \end{array}\right]  = 
\left[\tiny\begin{array}{lllll|ll}
-0.80046&-7.7472&0&0&0&-13.0415&11.9996\\
-1.9826&-16.5346&41.59&0&0&-25.1033&11.7281\\
0&-1.103&-2.6164&14.2226&0&-12.39&0.012639\\
0&0&-2.5597&-2.6421&6.1304&2.1379&0.89566\\
0&0&0&3.2099&-174.8766&0.90446&-2.2903\\ \hline
0.044385&0.23863&-1.6385&0.48079&2.0247&-2.0699\mathrm{e}{-5}&5.7173\mathrm{e}{-6}\\
\end{array}\right] \\
{\tiny K_{\rm blue}} &= \left[\begin{array}{l|l}A_K & B_K \\\hline C_K & D_K \end{array} \right]  =  \left[\tiny\begin{array}{lllll|ll}
-0.61907&-1.1401&0&0&0&0.25637&0.17058\\
16.7706&-4.1928&-1.523&0&0&-1.2753&0.43077\\
0&8.1615&-6.3251&-1.5961&0&-0.40252&-0.56916\\
0&0&-1.7351&-27.1582&-4.1308&-1.3393&5.0511\\
0&0&0&-17.2811&-83.1511&5.0909&3.4138\\ \hline
-11.1263&3.7925&-1.4411&-2.7711&3.9251&-9.9964\mathrm{e}{-5}&-2.355\mathrm{e}{-6}\\
\end{array}\right]
\end{split}
\end{align}

\begin{figure}[h]
\includegraphics[scale=0.45]{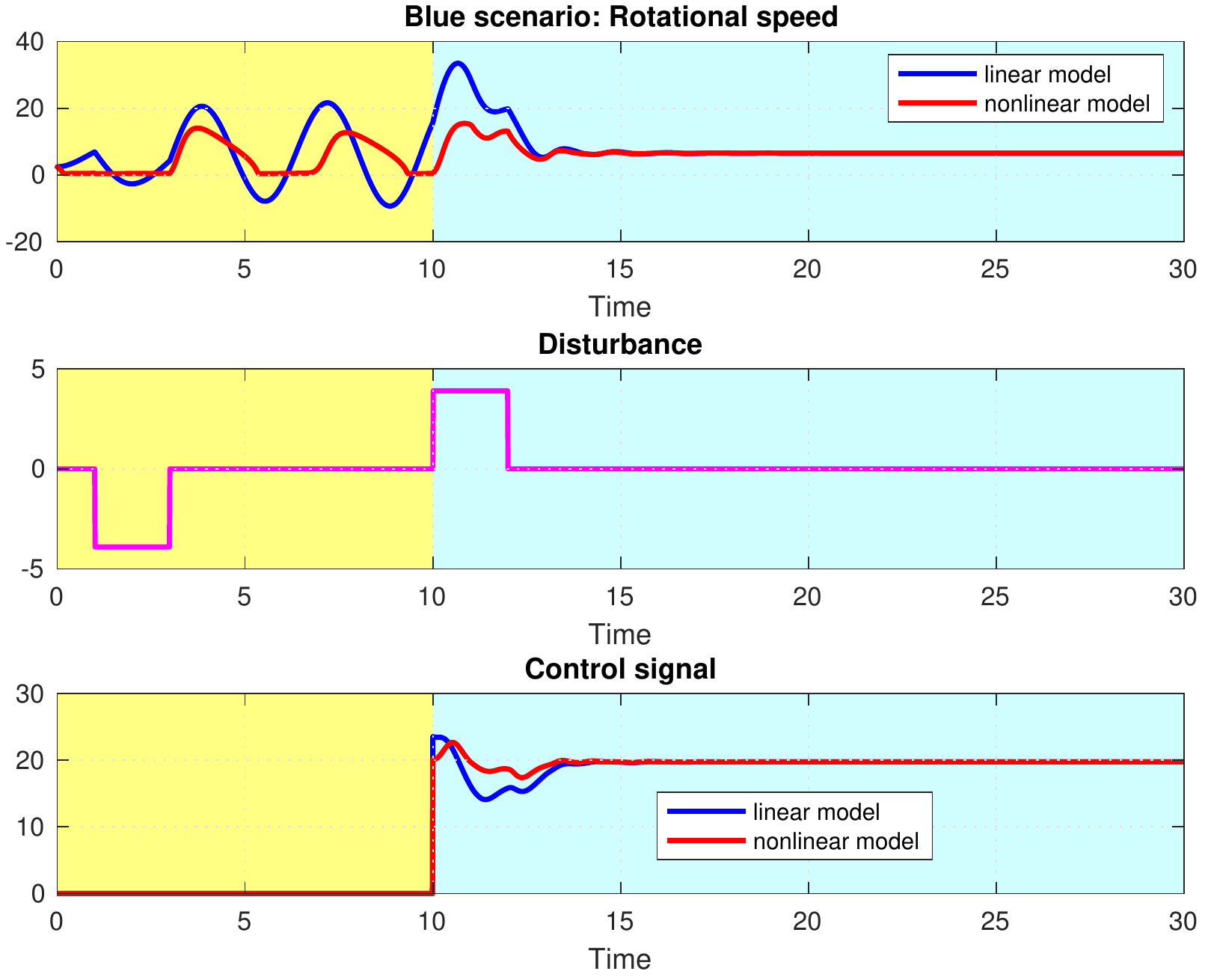}
\includegraphics[scale=0.45]{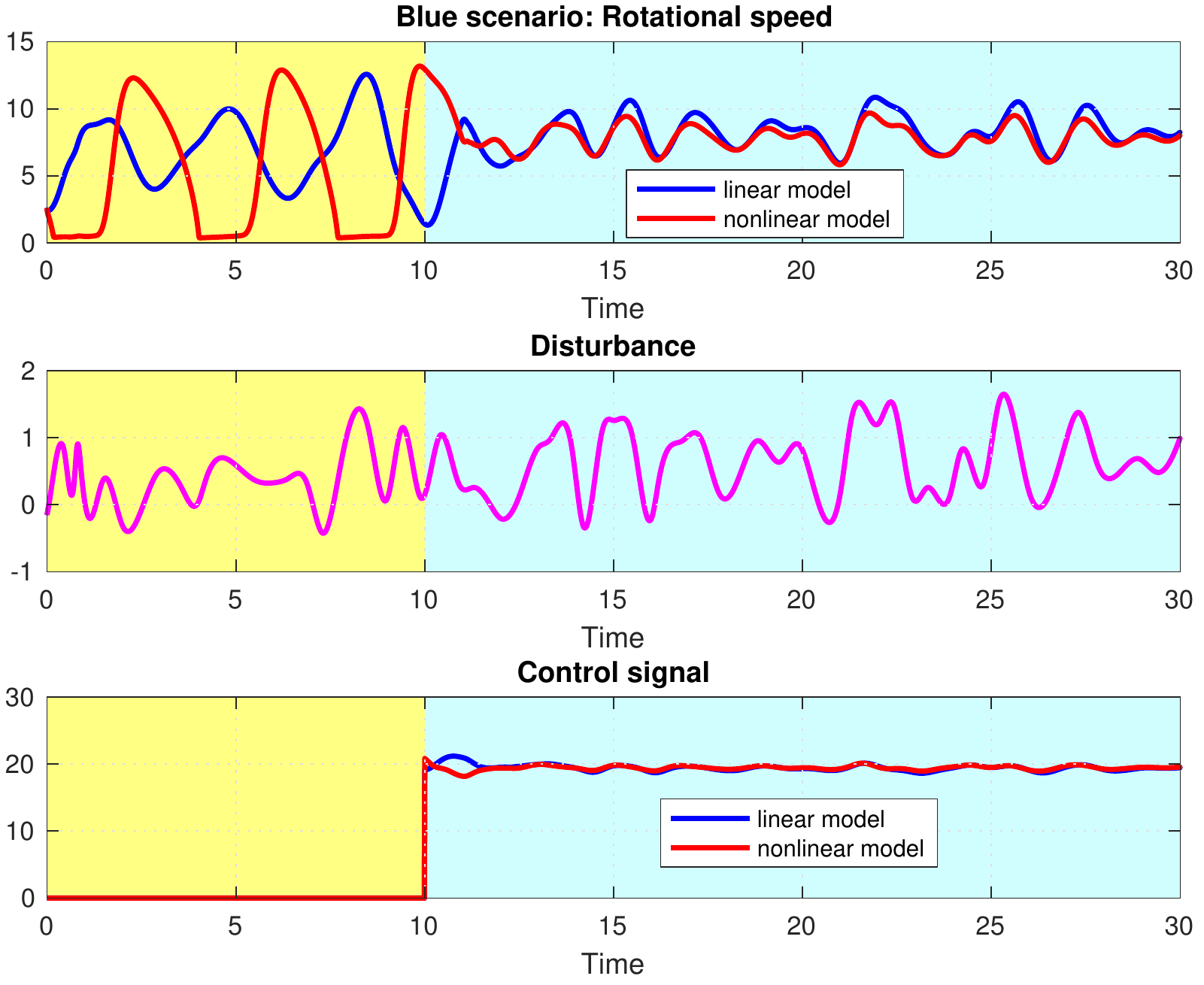}
\caption{Blue scenario: slip-stick caused by large disturbances. Stabilizing feedback with $K_{\rm blue}$ allows the rotational speed to recover. \label{blue_ondulating}}
\end{figure}


\subsection{Blue scenario}
The blue scenario is more challenging as the damping parameter $\lambda$ is between the two critical
values $\lambda_1(\alpha,q) < \lambda < \lambda_2(\alpha,q)$, giving rise to two
unstable poles. Here
slip-stick occurs quickly in open loop (Fig. \ref{openloop}).
While stabilization of the linear closed loop is based on the results
of section \ref{stabilize}, leading to a locally exponentially stable
non-linear closed loop, a global certificate via the sector non-linearity
(\ref{program1}) fails due to the very large primal sector in the blue case. In response, we use
the large magnitude sector constraint in tandem with  overshoot mitigation. Moreover,
a heuristic for the peak-to-peak norm is used, which leads to the mixed program
\begin{eqnarray}
\begin{array}{ll}
\mbox{minimize} & \| T_{y_1w}(K)\|_\infty \\
\mbox{subject to} & \|\widetilde{T}_{ze}(K)\|_2 \leq \rho(r) \\
& \|W_u T_{uw}(K)\|_\infty \leq 1  \\
&K \in \mathscr K_5,
\end{array}
\end{eqnarray}
the parameters now being $q_l = -3$, $q_u = -0.1$ and  $W_u(s) = \frac{1\mathrm{e}4 s }{s + 2\mathrm{e}5}$. 
 
The idea is to employ the $H_2$-norm of the LTI-system in Fig. \ref{zames_figure} as an indirect means
to reduce $\|\widetilde{T}_{ze}(K)\|_{{\rm pk}\_{\rm gn}}$, 
which amounts to replacing the $L_1$-norm of the impulse response by its energy. 
The parameter $\rho(r)$ has been estimated using trial and error  
so that the $H_2$ constraint ensures satisfaction of the peak-gain constraint in program (\ref{pgL1})
with parameter $r$. Starting again with
$K_0\in \mathscr K_5$ synthesized for a finite-difference model with $N=50$, we can then certify exponential stability and $H_\infty$-performance
of the infinite-dimensional loops
$(P,K_0)$ and $(\widetilde{P},K_0)$ via  \cite{AN:18}, and the $H_2$-certificate
with \cite[Lemma 3]{apkarian:19}. This controller is further
optimized in the true infinite dimensional system using the method of
\cite{AN:18,apkarian:19}, leading to the final $K_{\rm blue}\in \mathscr K_5$ in (\ref{controllers}). Posterior certification
shows that $\|\widetilde{T}_{ze}(K_{\rm blue})\|_2 < \rho(r) = 1.3$
implies $\|\widetilde{T}_{ze}(K_{\rm blue})\|_{{\rm pk}\_{\rm gn}}=0.680<r^{-1}= 1/1.45=0.690$, whereby the complementary large
magnitude sector condition is now satisfied in the discretized model with $N=200$. Infinite dimensional certification for
$\|\cdot\|_{{\rm pk}\_{\rm gn}}$ is currently not yet available,
even though this ought to be established along the lines of 
\cite[Lemma 4, Theorem 3]{AN:18} and \cite[Lemma 3]{apkarian:19}. The controller achieves
excellent results in the non-linear simulation. This is shown in Fig. \ref{test1} (right) where an initial condition generates slip-stick in open loop (yellow area). Triggering control at $t = 10$ removes slip-stick and additionally provides rejection against strong and sharp disturbances (blue area).   Similarly, in Fig. \ref{blue_ondulating}
the effect of switching the controller on is tested on two
different disturbances.

It should be mentioned that other ways to address the non-linearity $\psi$ have been discussed.
In \cite{bresch_krstic:14} an adaptive controller for a time varying $q(t)$  was constructed, while
\cite{apkarian:19} discusses parametric robust control for $q \in [\underline{q},\overline{q}]$
as well as gain-scheduling of
$q(t)$ as further possibilities.

\section*{Conclusion}
We have presented a novel method to
design exponentially stabilizing 
regulators of simple implementable structure
for boundary control of 
a wave equation with non-linear boundary anti-damping.
Our results are illustrated in
control of torsional vibrations
in drilling systems, and two scenarios labeled 'gray'
and 'blue' are discussed in detail.
We show that in order to 
avoid slip-stick it is crucial to optimize $H_\infty$-performance of the loop. In particular,
reducing overshoot by way of $H_\infty$ minimization 
proved effective for the more challenging 'blue'
scenario. The 'gray' scenario had
previously been
discussed  in the literature, and here the substantial
improvement of our method over published work 
is that we can design finite-dimensional
exponentially stabilizing 
controllers, which in addition show excellent performance. The 'blue' scenario is new and more
challenging due to inherent instability. We design
finite-dimensional controllers which stabilize the wave
equation locally exponentially, mitigate the slip-stick
effect, and in addition, have
a global boundedness certificate, based on the novel
concept of a large magnitude sector non-linearity.

\end{document}